\documentclass[11pt,a4paper]{amsart}

\usepackage{mathtools}
\usepackage{graphicx}
\usepackage{amsthm}
\usepackage{amssymb}
\usepackage{caption}
\usepackage{hyperref}
\usepackage{nicematrix}
\usepackage{enumitem}
\usepackage{booktabs}
\usepackage[dvipsnames]{xcolor}
\usepackage[left=30mm, right=30mm]{geometry}

\usepackage{tikz}
\usepackage{tikz-cd}
\usetikzlibrary{calc}
\usetikzlibrary{decorations.pathmorphing}
\usetikzlibrary{decorations.pathreplacing}

\theoremstyle{definition}
\newtheorem{definition}{Definition}[section]

\newtheorem{example}[definition]{Example}

\theoremstyle{plain}
\newtheorem{corollary}[definition]{Corollary}
\newtheorem{lemma}[definition]{Lemma}
\newtheorem{proposition}[definition]{Proposition}
\newtheorem{theorem}[definition]{Theorem}

\numberwithin{equation}{section}

\tikzset{
	vertex/.style={
		circle,
		minimum size=1.8mm,
		fill,
		inner sep=0,
		outer sep=0,
	},
	edge/.style={
		line width=.25mm,
	}
}

\newcommand*\defterm{\emph}

\newcommand*\diam[1]{\mathrm{diam}(#1)}
\newcommand*\dist[3]{d_{#1}(#2,#3)}
\newcommand*\girth[1]{\mathrm{girth}(#1)}
\newcommand*\knitdegree[1]{\mathrm{kd}(#1)}
\newcommand*\starknitdegree[1]{\mathrm{kd}^*(#1)}
\newcommand*\cliquenumber[2][]{\omega\parens[#1]{#2}}
\newcommand*\chromaticnumber[2][]{\chi\parens[#1]{#2}}

\newcommand\bigboxtimesaux[2]{%
	\mbox{%
		\begin{tikzpicture}[
			line join=round,
			line width=.5pt, 
			baseline=#2,
			]
			\pgfmathsetlengthmacro{\t}{#1}
			\draw (-\t,-\t) -- (\t,\t);
			\draw (-\t,\t) -- (\t,-\t);
			\draw (-\t,-\t) rectangle (\t,\t);
		\end{tikzpicture}%
	}
}

\newcommand\bigboxtimes{%
	\mathop{
		\mathchoice 
		{\bigboxtimesaux{.6em}{-.25em}}
		{\bigboxtimesaux{.45em}{-.25em}\,}
		{\boxtimes}
		{\boxtimes}
	}%
}

\newcommand\bignabla{%
	\mathop{
		\mathchoice 
		{\raisebox{-2pt}{\text{\LARGE$\nabla$}}}
		{\raisebox{-1.5pt}{\text{\Large$\nabla$}}}
		{\nabla}
		{\nabla}
	}%
}

\newcommand*\graphjoin{\bignabla}
\newcommand*\graphjointwo{\mathbin{\nabla}}
\newcommand*\strongprod{\bigboxtimes}

\newcommand*\commgraph[1]{\mathcal{G}(#1)}
\newcommand*\extendedcommgraph[1]{\mathcal{G}^*(#1)}

\newcommand*\centre[1]{Z(#1)}

\DeclarePairedDelimiter{\abs}{\lvert}{\rvert}

\DeclarePairedDelimiter{\parens}{\lparen}{\rparen}
\DeclarePairedDelimiter{\bracks}{\lbrack}{\rbrack}
\DeclarePairedDelimiter{\braces}{\{}{\}}

\DeclarePairedDelimiter{\set}{\{}{\}}
\DeclarePairedDelimiterX{\gset}[2]{\{}{\}}{\,#1:#2\,}

\newcommand*\Xn{\set{1,\ldots,n}}
\newcommand*\X[1]{\set{1,\ldots,#1}}
\newcommand*\setX{\Xn}

\newcommand*\NC{\mathit{NC}}

\newcommand*{\Tr}[1]{\mathcal{T}_{#1}}

\allowdisplaybreaks

\begin{document}

\title{Commuting graphs and semigroup constructions}

\author{Tânia Paulista}
\address[T. Paulista]{%
Center for Mathematics and Applications (NOVA Math) \& Department of Mathematics\\
NOVA School of Science and Technology\\
NOVA University of Lisbon\\
2829--516 Caparica\\
Portugal
}
\email{%
tpl.paulista@gmail.com
}
\thanks{This work is funded by national funds through the FCT -- Fundação para a Ciência e a Tecnologia, I.P., under the scope of the projects UID/297/2025 and UID/PRR/297/2025 (Center for Mathematics and Applications - NOVA Math). The author is also funded by national funds through the FCT -- Fundação para a Ciência e a Tecnologia, I.P., under the scope of the studentship 2021.07002.BD}

\thanks{The author is thankful to her supervisors António Malheiro and Alan J. Cain for all the support, encouragement and guidance; and also for reading a draft of this paper}


\begin{abstract}
	The aim of this paper is to see how commuting graphs interact with two semigroup constructions: the zero-union and the direct product. For both semigroup constructions, we investigate the diameter, clique number, girth, chromatic number and knit degree of their commuting graphs and, when possible, we exhibit the relationship between each one of these properties and the corresponding properties of the commuting graphs of the original semigroups.
\end{abstract}

\maketitle

\section{Introduction}

The commuting graph of a semigroup is a simple graph, contructed from a semigroup, that describes commutativity of elements. Commuting graphs were introduced in 1955 by Brauer and Fowler \cite{First_paper_commuting_graphs} and, since then, they have been widely studied. The close relationship between the algebraic structure of a semigroup and the combinatorial structure of its commuting graph contributes to the attention these graphs continue to receive. Moreover, this relationship makes these graphs useful tools to approach group/semigroup theoretical questions. For example, they played an important role in the discovery of three sporadic simple groups (now known as the Fischer groups) \cite{Sporadic_simple_groups}. Commuting graphs were also involved in the determination of an upper bound for the size of the abelian subgroups of a finite group \cite{Importance_commuting_graphs_1}. In addition, commuting graphs had an important role in proving various results concerning finite dimensional division algebras \cite{Importance_commuting_graphs_2, Importance_commuting_graphs_3, Importance_commuting_graphs_5, Importance_commuting_graphs_6, Importance_commuting_graphs_4}. Furthermore, they were used to answer (positively, except in one case) a conjecture formulated by Schein (see \cite{Schein_conjecture}) in the context of characterizing $r$-semisimple bands \cite{Commuting_graph_T_X}.

Commuting graphs have been studied from different perspectives. Several authors investigated the commuting graphs of important groups and semigroups, such as the symmetric group \cite{Commuting_graph_I_X, Symmetric_group, Diameter_commuting_graph_symmetric_group, Commuting_graph_symmetric_alternating_groups}, the alternating group \cite{Commuting_graph_symmetric_alternating_groups, Alternating_group}, the transformation semigroup \cite{Commuting_graph_T_X, Largest_commutative_T_X_P_X}, the symmetric inverse semigroup \cite{Commuting_graph_I_X} and the partial transformation semigroup \cite{Largest_commutative_T_X_P_X, Diameter_P_X}. Other authors focused on identifying which simple graphs are isomorphic to commuting graphs of groups/semigroups \cite{Graphs_arise_as_commuting_graphs_groups, Graphs_that_arise_as_commuting_graphs_of_semigroups, Graphs_that_arise_as_commuting_graphs_of_semigroups_2}. Another way to study commuting graphs is through the characterization of the groups/semigroups whose commuting graph has a certain property (such as being a cograph, a chordal graph, a perfect graph or a split graph) \cite{Graphs_arise_as_commuting_graphs_groups, Commuting_graphs_perfect, Commuting_graphs_groups_split}. Additionally, there are several papers \cite{Commuting_graph_T_X, Graphs_that_arise_as_commuting_graphs_of_semigroups, Graphs_that_arise_as_commuting_graphs_of_semigroups_2, Group_whose_commuting_graph_has_diameter_n, Completely_0-simple_paper, Completely_simple_semigroups_paper, Commuting_graphs_inverse_completely_regular} that address the following question: given a property of commuting graphs (such as the diameter, clique number, girth, chromatic number, knit degree) a class of semigroups $\mathcal{C}$ (for example, the class of groups, semigroups, completely simple semigroups, completely $0$-simple semigroups, inverse semigroups, completely regular semigroups) and $n\in\mathbb{N}$, is it possible to find a semigroup in the class $\mathcal{C}$ such that the chosen property of the commuting graph of that semigroup is equal to $n$?

In this paper we investigate commuting graphs from a different perspective: we aim to understand how commuting graphs interact with semigroup constructions. We conduct this study through the comparison of several properties of the commuting graph of a particular semigroup construction with the corresponding properties of the commuting graphs of the initial semigroups. This line of reasoning is motivated by the existence of several properties that are preserved by considering semigroup constructions: the preservation (or non-preservation) of properties of semigroups under various constructions has long been a subject of study \cite{Prop_smg_construction_3, Prop_smg_construction_4, Prop_smg_construction_2, Prop_smg_construction_1}. Thus it is natural to consider the analogous questions for commuting graphs of semigroup constructions. There has already been some work related to this topic: recently, the present author investigated the commuting graphs of Rees matrix semigroups over groups \cite{Completely_simple_semigroups_paper} and of $0$-Rees matrix semigroups over groups \cite{Completely_0-simple_paper}. In this paper we contribute to this topic by considering the commuting graphs of two other semigroup constructions: the zero-union of semigroups and the direct product of semigroups. The latter is a well-known semigroup construction that requires no introduction. The former yields the semigroup given by the disjoint union of all the original semigroups with a new element $0$, which inherits multiplication within the original semigroups and with the remaining products equal to $0$. Due to its simplicity, zero-unions of semigroups are frequently used in semigroup theory as a tool to construct examples/counterexamples. They are also useful to establish new results: in \cite{Zero_union_importance} they were used to prove a theorem regarding the concept of index for semigroups, and in the upcoming paper \cite{Commuting_graphs_inverse_completely_regular} they will be important in establishing, for each odd integer $n$ greater than $3$, the existence of a Clifford semigroup whose commuting graph has clique and chromatic numbers both equal to $n$.

The structure of the paper is as follows. In Section~\ref{Preliminaries} we gather several basic notions from graph theory that we will use in the paper. Moreover, we introduce the notions of commuting graphs and extended commuting graphs of semigroups. In Sections~\ref{cha: G(zero-union)}~and~\ref{cha: G(direct product)} we investigate the commuting graph of a zero-union of semigroups and of a direct product of semigroups, respectively. We are interested in studying the knit degree of these semigroup constructions, as well as the diameter, clique number, girth and chromatic number of their commuting graphs. Furthermore, we will see that several of these properties can be obtained from the corresponding properties of the commuting graphs of the original semigroups. In these cases, we exhibit the relationship between the properties of the relevant commuting graphs.


This paper is based on Chapters 6 and 7 of the author's Ph.D. thesis \cite{My_thesis}.

\section{Preliminaries} \label{Preliminaries}

\subsection{Simple graphs}\label{sec: graphs}

A \defterm{simple graph} $G=(V,E)$ consists of a non-empty set $V$ --- whose elements are called \defterm{vertices} --- and a set $E$ --- whose elements are called \defterm{edges} --- formed by $2$-subsets of $V$. Throughout this subsection we will assume that $G=(V,E)$ is a simple graph.

Let $x$ and $y$ be vertices of $G$. If $\set{x,y}\in E$, then we say that the vertices $x$ and $y$ are \defterm{adjacent}. If $\set{x,z}\notin E$ for all $z\in V$ (that is, if $x$ is not adjacent to any other vertex), then we say that $x$ is an \defterm{isolated vertex}.


If $H=\parens{V',E'}$ is also a simple graph, then we say that $G$ and $H$ are \defterm{isomorphic} if there exists a bijection $\varphi: V\to V'$ such that for all $x,y\in V$ we have $\set{x,y}\in E$ if and only if $\set{x\varphi,y\varphi}\in E'$ (that is, for all $x,y\in V$ we have that $x$ and $y$ are adjacent in $G$ if and only if $x\varphi$ and $y\varphi$ are adjacent in $H$).


A simple graph $H=\parens{V',E'}$ is a \defterm{subgraph} of $G$ if $V'\subseteq V$ and $E'\subseteq E$. Note that, since $H$ is a simple graph, the elements of $E'$ are $2$-subsets of $V'$.

Given $V'\subseteq V$, the \defterm{subgraph induced by $V'$} is the subgraph of $G$ whose set of vertices is $V'$ and where two vertices are adjacent if and only if they are adjacent in $G$ (that is, the set of edges of the induced subgraph is $\braces{\braces{x,y}\in E: x,y\in V'}$).

A \defterm{complete graph} is a simple graph where all distinct vertices are adjacent to each other. The unique (up to isomorphism) complete graph with $n$ vertices is denoted $K_n$.

A \defterm{null graph} is a simple graph with no edges and where all vertices are isolated vertices.


A \defterm{path} in $G$ from a vertex $x$ to a vertex $y$ is a sequence of pairwise distinct vertices (except, possibly, $x$ and $y$) $x=x_1,x_2,\ldots,x_n=y$ such that $\braces{x_1,x_2}, \braces{x_2,x_3},\ldots, \braces{x_{n-1},x_n}$ are pairwise distinct edges of $G$. The \defterm{length} of the path is the number of edges of the path; thus, the length of our example path is $n-1$. If $x=y$ then we call the path a \defterm{cycle}. Whenever we want to mention a path, we will write that $x=x_1-x_2-\cdots-x_n=y$ is a path (instead of writing that $x=x_1,x_2,\ldots,x_n=y$ is a path).

If $x$ and $y$ are vertices of $G$, then we are going to use the notation $x\sim y$ to mean that either $x=y$ or $\set{x,y}\in E$. Note that if $x_1-x_2-\cdots-x_n$ is a path, then we have $x_1\sim x_2 \sim\cdots\sim x_n$. However, if we have $x_1\sim x_2 \sim\cdots\sim x_n$, then that sequence of vertices does not necessarily form a path because there might exist distinct $i,j\in\Xn$ such that $x_i=x_j$.

We say that $G$ is \defterm{connected} if for all vertices $x,y\in V$ there is a path from $x$ to $y$.

The \defterm{distance} between two vertices $x$ and $y$, denoted $\dist{G}{x}{y}$, is the length of a shortest path from $x$ to $y$. If there is no such path between the vertices $x$ and $y$, then the distance between $x$ and $y$ is defined to be infinity, that is, $\dist{G}{x}{y}=\infty$. The \defterm{diameter} of $G$, denoted $\diam{G}$, is the maximum distance between vertices of $G$, that is, $\diam{G}=\max\gset{\dist{G}{x}{y}}{x,y\in V}$. We notice that the diameter of $G$ is finite if and only if $G$ is connected.

Let $K\subseteq V$. We say that $K$ is a \defterm{clique} in $G$ if $\braces{x,y}\in E$ for all $x,y\in K$, that is, if the subgraph of $G$ induced by $K$ is complete. The \defterm{clique number} of $G$, denoted $\cliquenumber{G}$, is the size of a largest clique in $G$, that is, $\cliquenumber{G}=\max\left\{|K|: K \text{ is a clique in } G\right\}$.

If the graph $G$ contains cycles, then the \defterm{girth} of $G$, denoted $\girth{G}$, is the length of a shortest cycle in $G$. If $G$ contains no cycles, then $\girth{G}=\infty$.

The \defterm{chromatic number} of $G$, denoted $\chromaticnumber{G}$, is the minimum number of colours required to colour the vertices of $G$ in a way such that adjacent vertices have different colours.

Let $G=(V,E)$ and $H=(V',E')$ be two simple graphs. We can assume, without loss of generality, that $V\cap V'=\emptyset$. In what follows we describe two graph operations that will be useful for characterizing commuting graphs in the next two sections.

The \defterm{graph join} of $G$ and $H$, denoted $G\graphjointwo H$, is defined to be the (simple) graph whose set of vertices is $V\cup V'$ and whose set of edges is $E\cup E'\cup\gset{\set{x,y}}{x\in V \text{ and } y\in V'}$. This means that, in the graph $G\graphjointwo H$, two vertices $x,y\in V\cup V'$ are adjacent if and only if one of the following conditions is satisfied:
\begin{enumerate}
	\item $x\in V$ and $y\in V'$ (or vice versa).
	\item $x,y\in V$ and $\braces{x,y}\in E$ (or $x,y\in V'$ and $\braces{x,y}\in E'$).
\end{enumerate}
It is straightforward to see that the graph join is an associative operation (in the sense that, if $G_1, G_2, G_3$ are simple graphs, then $\parens{G_1\graphjointwo G_2}\graphjointwo G_3$ is isomorphic to $G_1 \graphjointwo \parens{G_2\graphjointwo G_3}$). Furthermore, if $n\in\mathbb{N}$ and $G_i=\parens{V_i,E_i}$ is a simple graph for all $i\in\Xn$, then their graph join $\graphjoin_{i=1}^n G_i$ is (up to isomorphism) the graph with vertex set $\bigcup_{i=1}^n V_i$ and where two vertices $x$ and $y$ are adjacent if and only if one of the following conditions holds:
\begin{enumerate}
	\item There exist distinct $i,j\in\Xn$ such that $x\in V_i$ and $y\in V_j$.
	\item There exists $i\in\Xn$ such that $x,y\in V_i$ and $\braces{x,y}\in E_i$.
\end{enumerate}
This means that $\graphjoin_{i=1}^n G_i$ can be obtained from the graphs $G_1,\ldots,G_n$ by making all of the vertices of $G_i$ adjacent to all of the vertices of $G_j$ for all distinct $i,j\in\Xn$. 

The next lemma, which is easy to prove, shows the relationship between the clique and chromatic numbers of two graphs and of their graph join.  

\begin{lemma}\label{preli: graph join clique/chromatic numbers}
	Let $G$ and $H$ be two simple graphs. Then
	\begin{enumerate}
		\item $\cliquenumber{G\graphjointwo H}=\cliquenumber{G}+\cliquenumber{H}$.
		
		\item $\chromaticnumber{G\graphjointwo H}=\chromaticnumber{G}+\chromaticnumber{H}$.
	\end{enumerate}
\end{lemma}

The \defterm{strong product} of $G$ and $H$, denoted $G\boxtimes H$ is the (simple) graph whose set of vertices is $V\times V'$ and where two vertices $\parens{x_1,x_2}$ and $\parens{y_1,y_2}$ are adjacent if and only if one of the following three conditions is satisfied:
\begin{enumerate}
	\item $x_1=y_1$ and $\set{x_2,y_2}\in E'$.
	\item $\set{x_1,y_1}\in E$ and $x_2=y_2$.
	\item $\set{x_1,y_1}\in E$ and $\set{x_2,y_2}\in E'$.
\end{enumerate}
If we use the notation introduced above, then we have that $\parens{x_1,x_2}$ and $\parens{y_1,y_2}$ are adjacent if and only if $\parens{x_1,x_2}\neq\parens{y_1,y_2}$, $x_1\sim y_1$ (in $G$) and $x_2\sim y_2$ (in $H$). It is easy to see that the strong product of graphs is an associative operation (in the sense that, if $G_1, G_2, G_3$ are simple graphs, then $\parens{G_1\boxtimes G_2}\boxtimes G_3$ and $G_1\boxtimes\parens{G_2\boxtimes G_3}$ are isomorphic). Furthermore, if $n\in\mathbb{N}$ and $G_i=\parens{V_i,E_i}$ is a simple graph for all $i\in\Xn$, then their strong product $\strongprod_{i=1}^n G_i$ is (up to isomorphism) the graph with vertex set $\prod_{i=1}^n V_i$ and where two vertices $\parens{x_1,\ldots,x_n}$ and $\parens{y_1,\ldots,y_n}$ are adjacent if and only if $\parens{x_1,\ldots,x_n}\neq \parens{y_1,\ldots,y_n}$ and $x_i\sim y_i$ (in $G_i$) for all $i\in\Xn$.

The next lemma, which is easy to prove, provides a way to determine the clique number of the strong product of two graphs, as well as an upper bound for its chromatic number, using the clique and chromatic numbers, respectively, of the two graphs. 

\begin{theorem}\label{preli: strong prod clique/chromatic numbers}
	Let $G=(V,E)$ and $H=(V',E')$ be two simple graphs. Then
	\begin{enumerate}
		\item $\cliquenumber{G\boxtimes H}=\cliquenumber{G}\cdot\cliquenumber{H}$.
		
		\item $\chromaticnumber{G\boxtimes H}\leqslant\chromaticnumber{G}\cdot\chromaticnumber{H}$.
	\end{enumerate}
\end{theorem}

\subsection{Commuting graphs and extended commuting graphs}
\label{sec: (extended) commgraph}

The \defterm{center} of a semigroup $S$ is the set
\begin{displaymath}
	Z(S)= \left\{x\in S: xy=yx \text{ for all } y\in S\right\}.
\end{displaymath}

Let $S$ be a finite non-commutative semigroup. The \defterm{commuting graph} of $S$, denoted $\commgraph{S}$, is the simple graph whose set of vertices is $S\setminus Z(S)$ and where two distinct vertices $x,y\in S\setminus Z(S)$ are adjacent if and only if $xy=yx$.

Let $S$ be a finite semigroup. The \defterm{extended commuting graph} of $S$, denoted $\extendedcommgraph{S}$, is the simple graph whose set of vertices is $S$ and where two distinct vertices $x,y\in S$ are adjacent if and only if $xy=yx$. (Some authors use this definition for commuting graphs, instead of the one presented in the previous paragraph. See, for instance, \cite{Graphs_arise_as_commuting_graphs_groups, Cameron_commuting_graphs_notes, Commuting_graphs_groups_split}.)

It follows from both definitions that, for all vertices $x$ and $y$ of $\commgraph{S}$ (respectively $\extendedcommgraph{S}$), we have $x\sim y$ if and only if $xy=yx$.

Note that in the first definition the semigroup must be non-commutative (because otherwise we would obtain an empty vertex set), but in the second one we allow the semigroup to be commutative. Furthermore, as a consequence of the first definition we have $\diam{\commgraph{S}}\geqslant 2$ because, since $S$ must be non-commutative, then there exist $x,y\in S$ such that $xy\neq yx$, which implies that $\diam{\commgraph{S}}\geqslant\dist{\commgraph{S}}{x}{y}>1$. Additionally, the second definition implies that the center of the semigroup is a clique in the extended commuting graph of the semigroup.

The next lemma, which is easy to prove, gives a characterization of the extended commuting graph of a semigroup. When the semigroup is not commutative, this characterization shows a relationship between the commuting graph and the extended commuting graph of the semigroup.

\begin{lemma}\label{preli: commgraph and extended commgraph}
	Let $S$ be a finite semigroup.
	\begin{enumerate}
		\item If $S$ is commutative, then $\extendedcommgraph{S}$ is isomorphic to $K_{\abs{S}}$.
		
		\item If $S$ is non-commutative, then $\extendedcommgraph{S}$ is isomorphic to $K_{\abs{Z\parens{S}}}\graphjointwo\commgraph{S}$.
	\end{enumerate}
\end{lemma}

The notions of left path and knit degree, which we define below, were introduced in \cite{Commuting_graph_T_X} to settle a conjecture (posed by Schein \cite{Schein_conjecture}) concerning the characterization of $r$-semisimple brands.

Let $S$ be a non-commutative semigroup. A \defterm{left path} in $\commgraph{S}$ is a path $x_1,\ldots,x_n$ in $\commgraph{S}$ such that $x_1\neq x_n$ and $x_1x_i=x_nx_i$ for all $i\in\braces{1,\ldots,n}$. If $\commgraph{S}$ contains left paths, then the \defterm{knit degree} of $S$, denoted $\knitdegree{S}$, is the length of a shortest left path in $\commgraph{S}$.

We now extend the concepts of left path and knit degree to extended commuting graphs of semigroups, and we will call them $*$-left path and $*$-knit degree instead. This new definition will be useful in Section~\ref{cha: G(direct product)} for deducing the knit degree of the commuting graph of a direct product of semigroups.

Let $S$ be a semigroup. A \defterm{$*$-left path} in $\extendedcommgraph{S}$ is a path $x_1,\ldots,x_n$ in $\extendedcommgraph{S}$ such that $x_1\neq x_n$ and $x_1x_i=x_nx_i$ for all $i\in\braces{1,\ldots,n}$. If $\extendedcommgraph{S}$ contains $*$-left paths, then the \defterm{$*$-knit degree} of $S$, denoted $\starknitdegree{S}$, is the length of a shortest $*$-left path in $\extendedcommgraph{S}$.

It is easy to see that, when $S$ is a non-commutative semigroup and $\commgraph{S}$ contains left paths, then $\extendedcommgraph{S}$ contains $*$-left paths and $\starknitdegree{S}\leqslant\knitdegree{S}$. The following lemma gives more information about $*$-left paths in $\extendedcommgraph{S}$ and the $*$-knit degree of $S$.

\begin{lemma}\label{preli: *-left paths and *-knit degree}
	\begin{enumerate}
		\item Suppose that $S$ is a commutative semigroup and $\extendedcommgraph{S}$ contains $*$-left paths. Then $\starknitdegree{S}=1$.
		
		\item Suppose that $S$ is a non-commutative semigroup and $\extendedcommgraph{S}$ contains left paths.
		\begin{enumerate}
			\item If $\extendedcommgraph{S}$ contains a $*$-left path that is not a left path in $\commgraph{S}$, then $\starknitdegree{S}\in\set{1,2}$.
			
			\item If all the $*$-left paths in $\extendedcommgraph{S}$ are left paths in $\commgraph{S}$, then $\starknitdegree{S}=\knitdegree{S}$.
		\end{enumerate}
	\end{enumerate}
\end{lemma}

\begin{proof}
	\textbf{Part 1.} Suppose that $S$ is a commutative semigroup and $\extendedcommgraph{S}$ contains $*$-left paths. Let $x_1-x_2-\cdots-x_n$ be a $*$-left path in $\commgraph{S}$. Then $x_1\neq x_n$, $x_1^2=x_nx_1$ and $x_1x_n=x_n^2$. Furthermore, we have $x_1x_n=x_nx_1$ because $S$ is commutative. Thus $x_1-x_n$ is a $*$-left path in $\extendedcommgraph{S}$ and, consequently, $\starknitdegree{S}=1$.
	
	\medskip
	
	\textbf{Part 2.} Suppose that $S$ is a non-commutative semigroup and that $\extendedcommgraph{S}$ contains left paths.
	
	Assume that $\extendedcommgraph{S}$ contains a $*$-left path that is not a left path in $\commgraph{S}$. Let $x_1-x_2-\cdots-x_n$ be such a $*$-left path in $\extendedcommgraph{S}$. Since $x_1-x_2-\cdots-x_n$ is not a left path in $\commgraph{S}$, then there exists $m\in\X{n}$ such that $x_m\in \centre{S}$. Hence $x_1x_m=x_mx_1$ and $x_nx_m=x_mx_n$. Furthermore, the fact that $x_1-x_2-\cdots-x_n$ is a $*$-left path in $\extendedcommgraph{S}$ implies that $x_1x_1=x_nx_1$ and $x_1x_n=x_nx_n$ and $x_1x_m=x_nx_m$. So, if $m\in\Xn\setminus\set{1,n}$, we have that $x_1-x_m-x_n$ is a $*$-left path (of length 2) in $\extendedcommgraph{S}$; and, if $m\in\set{1,n}$, we have that $x_1-x_n$ is a $*$-left path (of length 1) in $\extendedcommgraph{S}$. Thus $\starknitdegree{S}\leqslant 2$.
	
	
	
	Now assume that all the $*$-left paths in $\extendedcommgraph{S}$ are left paths in $\commgraph{S}$. Then $\knitdegree{S}\leqslant \starknitdegree{S}$. Additionally, by the paragraph before the lemma statement, we have $\starknitdegree{S}\leqslant\knitdegree{S}$, which concludes the proof.
\end{proof}

\section{The commuting graph of a zero-union}
\label{cha: G(zero-union)}

Let $n\in\mathbb{N}$. Let $S_1,\ldots,S_n$ be finite semigroups and let $S$ be their zero-union. We recall that a \defterm{zero-union} of $n$ semigroups $S_1,\ldots,S_n$, which we assume to be disjoint, is the set $\set{0}\cup\bigcup_{i=1}^n S_i$, where $0$ is a new element, and where the product of any two elements $x$ and $y$ is equal to the element $xy\in S_i$, if $x,y\in S_i$ for some $i\in\Xn$, and $0$ for the remaining cases. We partition $\Xn$ as $C\cup\NC$, where
\begin{align*}
	&C=\gset{i\in\Xn}{S_i \text{ is commutative}},\\
	&\NC=\gset{i\in\Xn}{S_i \text{ is not commutative}}. 
\end{align*}

The aim of this section is to study the graph $\commgraph{S}$ in terms of its properties and see if there is any relationship between them and the properties of $\commgraph{S_i}$ for all $i\in\NC$. We are going to determine the diameter, clique number, girth, chromatic number and knit degree.

\begin{proposition}\label{0Union: Z(S)}
	We have $\centre{S}=\set{0}\cup\bigcup_{i=1}^n\centre{S_i}$. Moreover, $S$ is commutative if and only if $S_i$ is commutative for all $i\in\Xn$.
\end{proposition}

\begin{proof}
	First we are going to prove that $\centre{S}\subseteq\set{0}\cup\bigcup_{i=1}^n\centre{S_i}$. Let $x\in\centre{S}$. We have $x\in\set{0}\cup\bigcup_{i=1}^n S_i$ and $xy=yx$ for all $y\in\set{0}\cup\bigcup_{i=1}^n S_i$. If $x=0$, then $x\in\set{0}\cup\bigcup_{i=1}^n S_i$. If $x\in S_i$ for some $i\in\Xn$, then it follows from the fact that $xy=yx$ for all $y\in S_i$ that $x\in\centre{S_i}\subseteq \set{0}\cup\bigcup_{i=1}^n \centre{S_i}$. Therefore $\centre{S}\subseteq\set{0}\cup\bigcup_{i=1}^n \centre{S_i}$.
	
	Now we prove the opposite inclusion. Let $i\in\Xn$ and $x\in\centre{S_i}$. Then $xy=yx$ for all $y\in S_i$. We also have $0x=0=x0$ and $xy=0=yx$ for all $j\in\Xn\setminus\set{i}$ and $y\in S_j$. Thus $x\in\centre{S}$. Additionally, it is clear that $0\in\centre{S}$. Therefore $\set{0}\cup\bigcup_{i=1}^n \centre{S_i}\subseteq\centre{S}$.
	
	Moreover, since $\set{0},S_1,\ldots,S_n$ are pairwise disjoint, we have
	\begin{align*}
		&S \text{ is commutative}\\
		\iff{} &\centre{S}=S\\
		\iff{} &\set{0}\cup\bigcup_{i=1}^n\centre{S_i}=\set{0}\cup\bigcup_{i=1}^n S_i\\
		\iff{} &\centre{S_i}=S_i\text{ for all }i\in\Xn\\
		\iff{} &S_i \text{ is commutative for all } i\in\Xn.\qedhere
	\end{align*}
\end{proof}

It follows from Proposition~\ref{0Union: Z(S)} that $S$ is not commutative if and only if $\NC\neq\emptyset$. In this situation, we have that due to the fact that $\set{0},S_1,\ldots,S_n$ are pairwise disjoint and $\centre{S_i}=S_i$ for all $i\in C$, the set of vertices of $\commgraph{S}$ is
\begin{displaymath}
	S\setminus\centre{S}=\parens[\bigg]{\set{0}\cup\bigcup_{i=1}^n S_i}\setminus\parens[\bigg]{\set{0}\cup\bigcup_{i=1}^n\centre{S_i}}=\bigcup_{i=1}^n\ S_i\setminus\centre{S_i}=\bigcup_{i\in\NC} S_i\setminus\centre{S_i}.
\end{displaymath}
This implies that the elements of the commutative semigroups (that is, the elements of $S_i$ for all $i\in C$) are not vertices of $\commgraph{S}$.

We consider two situations: $\abs{\NC}=1$ and $\abs{\NC}\geqslant 2$. In Theorem~\ref{0Union: S_i commutative, S_j non-commutative} we characterize $\commgraph{S}$ when we consider the former situation, and in Theorem~\ref{0Union: characterization S, S1 and S2 non-commutative} we characterize $\commgraph{S}$ when we consider the latter. Additionally, for the last case we also obtain the clique number (Corollary~\ref{0Union: clique number}), chromatic number (Corollary~\ref{0Union: chromatic number}), diameter (Corollary~\ref{0Union: diameter}), girth (Theorem~\ref{0Union: girth}) and knit degree (Theorem~\ref{0Union: knit degree}).

\begin{theorem}\label{0Union: S_i commutative, S_j non-commutative}
	Supppose that $\NC=\set{j}$. Then $\commgraph{S}=\commgraph{S_j}$.
\end{theorem}

\begin{proof}
	Since $\NC=\set{j}$, then
	\begin{displaymath}
		S\setminus\centre{S}= \bigcup_{i\in\NC} S_i\setminus\centre{S_i}=S_j\setminus\centre{S_j},
	\end{displaymath}
	which means that the set of vertices of $\commgraph{S}$ is equal to the set of vertices of $\commgraph{S_j}$. Furthermore, it is clear that given distinct $x,y\in S\setminus\centre{S}=S_j\setminus\centre{S_j}$, we have
	\begin{align*}
		& x \text{ and } y \text{ are adjacent in } \commgraph{S}\\
		\iff{} & xy=yx\\
		\iff{} & x \text{ and } y \text{ are adjacent in } \commgraph{S_j},
	\end{align*}
	which implies that the set of edges of $\commgraph{S}$ is equal to the set of edges of $\commgraph{S_j}$.
\end{proof}

As a consequence of Theorem~\ref{0Union: S_i commutative, S_j non-commutative} we have that, when $\NC=\set{j}$, then each one of the properties of $\commgraph{S}$ coincide with the corresponding properties of the $\commgraph{S_j}$. Furthermore, since $S_j\subseteq S$, it is also true that ($\commgraph{S}$ contains left paths if and only if $\commgraph{S_j}$ contains left paths and) $\knitdegree{S}=\knitdegree{S_j}$.

\begin{theorem}\label{0Union: characterization S, S1 and S2 non-commutative}
	Suppose that $\abs{\NC}\geqslant 2$. Then $\commgraph{S}=\graphjoin_{i\in \NC}\commgraph{S_i}$.
\end{theorem}

\begin{proof}
	In order to prove that $\commgraph{S}=\graphjoin_{i\in \NC}\commgraph{S_i}$ it is enough to verify that the set of vertices of $\commgraph{S}$ is equal to the union of the (disjoint) sets of vertices of $\commgraph{S_i}$ for all $i\in\NC$, that $\commgraph{S_i}$ is an induced subgraph of $\commgraph{S}$ for all $i\in\NC$, and that all the vertices of $\commgraph{S_i}$ are adjacent to all the vertices of $\commgraph{S_j}$ for all distinct $i,j\in\NC$.
	
	We have
	\begin{displaymath}
		S\setminus\centre{S}=\bigcup_{i\in\NC} S_i\setminus\centre{S_i}.
	\end{displaymath}
	Hence the set of vertices of $\commgraph{S}$ is equal to the union of the sets of vertices of $\commgraph{S_i}$ for all $i\in\NC$.
	
	Let $i\in\NC$ and $x,y\in S_i\setminus\centre{S_i}\subseteq S\setminus\centre{S}$ be such that $x\neq y$. We have
	\begin{align*}
		&x \text{ and } y \text{ are adjacent in } \commgraph{S}\\
		\iff{} & xy=yx\\
		\iff{} & x \text{ and } y \text{ are adjacent in } \commgraph{S_i}.
	\end{align*}
	Thus $\commgraph{S_i}$ is the subgraph of $\commgraph{S}$ induced by $S_i\setminus\centre{S_i}$.
	
	Let $i,j\in\Xn$ be such that $i\neq j$ and let $x\in S_i\setminus\centre{S_i}$ and $y\in S_j\setminus\centre{S_j}$. Since $xy=0=yx$, then $x$ and $y$ are adjacent in $\commgraph{S}$. This proves that all the vertices of $\commgraph{S_i}$ are adjacent to all the vertices of $\commgraph{S_j}$.
\end{proof}

Corollaries~\ref{0Union: clique number} and \ref{0Union: chromatic number} are direct consequences of Theorem~\ref{0Union: characterization S, S1 and S2 non-commutative} and (an iterated use of) Lemma~\ref{preli: graph join clique/chromatic numbers}. Furthermore, they establish a relationship between the clique number (respectively, chromatic number) of $\commgraph{S}$ and the clique numbers (respectively, chromatic numbers) of $\commgraph{S_i}$ for all $i\in\NC$.

\begin{corollary}\label{0Union: clique number}
	Suppose $\abs{NC}\geqslant 2$. Then $\cliquenumber{\commgraph{S}}=\sum_{i\in\NC}\cliquenumber{\commgraph{S_i}}$.
\end{corollary}

\begin{corollary}\label{0Union: chromatic number}
	Suppose $\abs{NC}\geqslant 2$. Then $\chromaticnumber{\commgraph{S}}=\sum_{i\in\NC}\chromaticnumber{\commgraph{S_i}}$.
\end{corollary}

\begin{corollary}\label{0Union: diameter}
	Suppose $\abs{NC}\geqslant 2$. Then $\commgraph{S}$ is connected and $\diam{\commgraph{S}}=2$.
\end{corollary}

\begin{proof}
	Let $x,y\in S\setminus\centre{S}=\bigcup_{i\in\NC} S_i\setminus\centre{S_i}$ be two vertices of $\commgraph{S}$. It follows from Theorem~\ref{0Union: characterization S, S1 and S2 non-commutative} that $\commgraph{S}=\graphjoin_{i\in\NC}\commgraph{S_i}$. Then, we have the following two cases:
	
	\smallskip
	
	\textit{Case 1:} Assume that there exist distinct $j,k\in\NC$ such that $x\in S_j\setminus\centre{S_j}$ and $y\in S_k\setminus\centre{S_k}$. Then $x$ is a vertex of $\commgraph{S_j}$ and $y$ is a vertex of $\commgraph{S_k}$. Thus $x\sim y$ (in $\commgraph{S}$) and, consequently, $\dist{\commgraph{S}}{x}{y}\leqslant 1$.
	
	\smallskip
	
	\textit{Case 2:} Assume that there exist $j\in\NC$ such that $x,y\in S_j\setminus\centre{S_j}$. Let $k\in\NC\setminus\set{j}$ and let $z\in S_k\setminus\centre{S_k}$ be a vertex of $\commgraph{S}$. We have that $x$ and $y$ are vertices of $\commgraph{S_j}$ and $z$ is a vertex of $\commgraph{S_k}$. Then $x\sim z \sim y$ (in $\commgraph{S}$) and, consequently, $\dist{\commgraph{S}}{x}{y}\leqslant 2$.
	
	\smallskip
	
	It follows from cases 1 and 2 that $\diam{\commgraph{S}}\leqslant 2$. Moreover, due to the fact that $\abs{\NC}\geqslant 2$, there exists $j\in\Xn$ such that $S_j$ is not commutative. Hence there exist $x,y\in S_j$ such that $xy\neq yx$ and, consequently, $x$ and $y$ are not adjacent in $\commgraph{S}$, which implies that $\diam{\commgraph{S}}\geqslant \dist{\commgraph{S}}{x}{y}>1$.
\end{proof}

In the next theorem we are going to see that, when $\abs{\NC}\geqslant 2$, then the girth of $\commgraph{S}$ does not depend on the girth of the graphs $\commgraph{S_i}$, $i\in\NC$ (unlike what happens with the clique and chromatic numbers of $\commgraph{S}$). Instead, it depends on $\abs{\NC}$ and whether there exists $i\in\NC$ such that $\commgraph{S_i}$ is not a null graph.

\begin{theorem}\label{0Union: girth}
	Suppose that $\abs{\NC}\geqslant 2$. Then $\commgraph{S}$ contains cycles. Moreover,
	\begin{enumerate}
		\item If $\abs{\NC}\geqslant 3$ or there exists $i\in \NC$ such that $\commgraph{S_i}$ is not a null graph, then $\girth{\commgraph{S}}=3$.
		
		\item If $\abs{\NC}=2$ and $\commgraph{S_i}$ is a null graph for all $i\in \NC$, then $\girth{\commgraph{S}}=4$.
	\end{enumerate}
\end{theorem}

\begin{proof}
	\textit{Case 1:} Suppose that $\abs{\NC}\geqslant 3$. Then there exist distinct $i,j,k\in\NC$. Let $x\in S_i\setminus\centre{S_i}$ be a vertex of $\commgraph{S_i}$, $y\in S_j\setminus\centre{S_j}$ be a vertex of $\commgraph{S_j}$ and $z\in S_k\setminus\centre{S_k}$ be a vertex of $\commgraph{S_k}$. Then, as a consequence of the characterization of $\commgraph{S}$ given by Theorem~\ref{0Union: characterization S, S1 and S2 non-commutative}, we have that $x$, $y$ and $z$ are vertices of $\commgraph{S}$ and they are adjacent to each other (in $\commgraph{S}$), which implies that $x-y-z-x$ is a cycle (of length 3) in $\commgraph{S}$. Thus $\girth{\commgraph{S}}=3$.
	
	\smallskip
	
	\textit{Case 2:} Suppose that there exists $i\in \NC$ such that $\commgraph{S_i}$ is not a null graph. Then there exist distinct $x,y\in S_i\setminus \centre{S_i}$ such that $x$ and $y$ are adjacent vertices of $\commgraph{S_i}$ (and, consequently, of $\commgraph{S}$). Let $j\in\NC\setminus\set{i}$ and let $z\in S_j\setminus\centre{S_j}$ be a vertex of $\commgraph{S_j}$. As a consequence of Theorem~\ref{0Union: characterization S, S1 and S2 non-commutative}, we have that $z$ is adjacent to $x$ and $y$ (in $\commgraph{S}$). Therefore, $x-y-z-x$ is a cycle (of length $3$) in $\commgraph{S}$ and, consequently, $\girth{\commgraph{S}}=3$.
	
	\smallskip
	
	\textit{Case 3:} Suppose that $\abs{\NC}=2$ and $\commgraph{S_i}$ is a null graph for all $i\in\NC$. Assume that $\NC=\set{i,j}$. Since $S_i$ and $S_j$ are not commutative, then there exist distinct $x,y\in S_i\setminus\centre{S_i}$ and distinct $z,w\in S_j\setminus\centre{S_j}$. If we have in mind the characterization of $\commgraph{S}$ given by Theorem~\ref{0Union: characterization S, S1 and S2 non-commutative}, then we can see that $x$ and $y$ are both adjacent to $z$ and $w$. Thus $x-z-y-w-x$ is a cycle (of length 4) in $\commgraph{S}$ and, consequently, $\girth{\commgraph{S}}\leqslant 4$.
	
	We only need to verify that $\commgraph{S}$ contains no cycles of length 3. Let $x_1-x_2-x_3-x_4$ be a path of length 3 in $\commgraph{S}$. It is enough to show that $x_1\neq x_4$. Assume, without loss of generality, that $x_2\in S_i\setminus\centre{S_i}$ (that is, $x_2$ is a vertex of $\commgraph{S_i}$). It follows from the fact that $\commgraph{S_i}$ is a null graph that there is no vertex of $\commgraph{S_i}$ that is adjacent to $x_2$. Thus $x_1,x_3\in S_j\setminus\centre{S_j}$ (that is, $x_1$ and $x_3$ are vertices of $\commgraph{S_j}$). Since $\commgraph{S_j}$ is also a null graph, then $x_4\in S_i\setminus\centre{S_i}$ is a vertex of $\commgraph{S_i}$. We just proved that $x_1\in S_j\setminus\centre{S_j}$ and $x_4\in S_i\setminus\centre{S_i}$, which implies that $x_1\neq x_4$. Thus $\commgraph{S}$ contains no cycles of length 3 and, as a consequence, we have $\girth{\commgraph{S}}=4$. 
\end{proof}

One of the necessary conditions for the commuting graph of a semigroup to contain left paths is the existence of distinct non-central elements $x$ and $y$ such that $x^2=yx$ and $y^2=yx$. Although in general this condition is not enough to guarantee the existence of left paths (see Example~\ref{0Union: example knit degree}), we will see in Theorem~\ref{0Union: knit degree} that this is true for zero-unions of semigroups (when $\abs{\NC}\geqslant 2$).

\begin{theorem}\label{0Union: knit degree}
	Suppose that $\abs{\NC}\geqslant 2$. Then $\commgraph{S}$ contains left paths if and only if there exist $i\in\NC$ and distinct $x,y\in S_i\setminus\centre{S_i}$ such that $x^2=yx$ and $y^2=xy$, in which case $\knitdegree{S}\in\set{1,2}$. Furthermore, $\knitdegree{S}=1$ if and only if there exists $i\in\NC$ such that $\commgraph{S_i}$ contains left paths and $\knitdegree{S_i}=1$.
\end{theorem}

\begin{proof}
	Suppose that $\commgraph{S}$ contains left paths. Let $x_1-x_2-\cdots-x_n$ be a left path in $\commgraph{S}$. Then $x_1\neq x_n$ and $x_1^2=x_nx_1$ and $x_n^2=x_1x_n$. Since $x_1$ and $x_n$ are vertices of $\commgraph{S}$, then $x_1,x_n\in \bigcup_{j\in\NC} S_j\setminus\centre{S_j}$. Let $i\in\NC$ be such that $x_1\in S_i\setminus\centre{S_i}$. Hence $x_nx_1=x_1^2\in S_i$, which implies, by the definition of a zero-union, that $x_n\in S_i\setminus\centre{S_i}$.
	
	Now suppose that there exist $i\in\NC$ and distinct $x,y\in S_i\setminus\centre{S_i}$ such that $x^2=yx$ and $y^2=xy$. Let $j\in\NC\setminus\set{i}$ and $z\in S_j\setminus\centre{S_j}$. Then $x$ and $y$ are vertices of $\commgraph{S_i}$ and $z$ is a vertex of $\commgraph{S_j}$. It follows from the characterization of $\commgraph{S}$ given by Theorem~\ref{0Union: characterization S, S1 and S2 non-commutative} that the vertex $z$ is adjacent to the vertices $x$ and $y$ (in $\commgraph{S}$). Hence $x-z-y$ is a path (of length 2) in $\commgraph{S}$. In addition, we have $x^2=xy$ and $xz=0=yz$ and $xy=y^2$. Hence $x-z-y$ is a left path in $\commgraph{S}$ and we have $\knitdegree{S}\leqslant 2$.
	
	The only thing left to do is to determine in which cases we have $\knitdegree{S}=1$. Suppose that $\knitdegree{S}=1$. This implies that there exist distinct $x,y\in S\setminus\centre{S}$ such that $x-y$ is a left path in $\commgraph{S}$. Then we have $xy=yx$ and $x^2=yx$ and $xy=y^2$. Since $x\in S\setminus\centre{S}=\bigcup_{j\in\NC} S_j\setminus\centre{S_j}$, then $x\in S_i\setminus\centre{S_j}$ for some $i\in\NC$. Hence $yx=x^2\in S_i$, and it follows from the definition of a zero-union that $y\in S_i\setminus\centre{S_i}$. Thus $x-y$ is a left path in $\commgraph{S_i}$ and $\knitdegree{S_i}=1$.
	
	Now suppose that there exists $i\in\NC$ such that $\commgraph{S_i}$ contains left paths and $\knitdegree{S_i}=1$. Hence there exist distinct $x,y\in S_i\setminus\centre{S_i}$ such that $x-y$ is a left path in $\commgraph{S_i}$. Since $\commgraph{S}=\graphjoin_{j\in\NC}\commgraph{S_j}$ (by Theorem~\ref{0Union: characterization S, S1 and S2 non-commutative}), then $x-y$ is also a left path in $\commgraph{S}$. Therefore $\knitdegree{S}=1$.
\end{proof}

\begin{example}\label{0Union: example knit degree}
	We consider the commuting graph of $\Tr{2}$ (the full transformation semigroup over $\set{1,2}$), which is shown in Figure~\ref{0Union, Figure: commuting graph of T_2}. We can see that $\commgraph{\Tr{2}}$ has no edges. Hence there are no paths of length greater than $1$ in $\commgraph{\Tr{2}}$, which implies that $\commgraph{\Tr{2}}$ contains no left paths. Nonetheless we have $\alpha^2=\alpha=\beta\alpha$ and $\beta^2=\beta=\alpha\beta$, where $\alpha=\begin{pmatrix}1&2\\ 1&1\end{pmatrix}$ and $\beta=\begin{pmatrix}1&2\\ 2&2\end{pmatrix}$.
		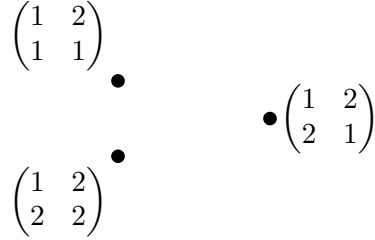
\begin{figure}[hbt]
				\begin{center}
						\begin{tikzpicture}
								
								\node[vertex] (11) at (0,0) {};
								\node[vertex] (22) at (0,-1) {};
								\node[vertex] (21) at (2,-0.5) {};
								
								\node[anchor=south east] at (11) {$\begin{pmatrix}1&2\\ 1&1\end{pmatrix}$};
								\node[anchor=north east] at (22) {$\begin{pmatrix}1&2 \\ 2&2\end{pmatrix}$};
								\node[anchor=west] at (21) {$\begin{pmatrix}1&2 \\ 2&1\end{pmatrix}$};
								
							\end{tikzpicture}
					\end{center}
				\caption{Commuting graph of the transformation semigroup $\Tr{2}$.}
				\label{0Union, Figure: commuting graph of T_2}
			\end{figure}
\end{example}

\section{The commuting graph of a direct product}
\label{cha: G(direct product)}

Let $n\in\mathbb{N}$ and let $S_1,\ldots,S_n$ be finite semigroups. We recall that the \defterm{direct product} of the semigroups $S_1,\ldots,S_n$ is the set $\prod_{i\in I} S_i$ (that is, the cartesian product of the semigroups $S_1,\ldots,S_n$) with componentwise multiplication: $\parens{i}\parens{st}=\parens{i}s\parens{i}t$ for all $s,t\in \prod_{i=1}^n S_i$.

Let $S=\prod_{i=1}^n S_i$. Given $s\in S$, we are going to denote the $i$-th component of $s$ by $s_i$, that is, $\parens{i}s=s_i$. We partition $\setX$ as $C\cup\NC$, where
\begin{align*}
	&C=\gset{i\in\setX}{S_i \text{ is commutative}},\\
	&\NC=\gset{i\in\setX}{S_i \text{ is not commutative}}. 
\end{align*}

The aim of this section is to determine the diameter, clique number, girth and chromatic number of $\commgraph{S}$, as well as the knit degree of $S$. Moreover, we want to find possible relations between the properties of $\commgraph{S}$ and the properties of $\commgraph{S_i}$ (for all $i \in \NC$) and $\extendedcommgraph{S_i}$ (for all $i \in C$), that is, we are going to see if we can obtain the properties of $\commgraph{S}$ by looking at the properties of $\commgraph{S_i}$ (for all $i\in\NC$) and $\extendedcommgraph{S_i}$ (for all $i\in C$).

The following lemma describes how commutativity works in $S$, that is, it provides a way to identify adjacent vertices in $\commgraph{S}$. It is immediate from the definition, but we state it here because we use it so often.

\begin{lemma}\label{direct prod: commutativity}
	Let $s,r\in S$. Then $sr=rs$ if and only if $s_ir_i=r_is_i$ for all $i\in\setX$.
\end{lemma}

\begin{proposition}\label{direct prod: Z(S), S comm <=> S_i comm}
	We have $\centre{S}=\prod_{i=1}^n \centre{S_i}$. Furthermore, $S$ is commutative if and only if $S_i$ is commutative for all $i\in\setX$.
\end{proposition}

\begin{proof}
	First we prove the inclusion $\centre{S}\subseteq\prod_{i=1}^n \centre{S_i}$. Let $s\in \centre{S}$. Let $j\in\setX$ and $x\in S_j$. Let $r\in S$ be such that for all $i\in \setX$
	\begin{displaymath}
		r_i=\begin{cases}
			s_i& \text{if } i\neq j,\\
			x& \text{if } i=j.
		\end{cases}
	\end{displaymath}
	We have $sr=rs$, which implies, by Lemma~\ref{direct prod: commutativity}, that $s_jx=s_jr_j=r_js_j=xs_j$. Since $x$ is an arbitrary element of $S_j$, we have $s_j\in\centre{S_j}$. Therefore $s\in \prod_{i=1}^n\centre{S_i}$.
	
	Now we verify the opposite inclusion. Let $s\in \prod_{i=1}^n\centre{S_i}$. Then $s_i\in\centre{S_i}$ for all $i\in\setX$. Let $r\in S$. We have $s_ir_i=r_is_i$ for all $i\in\setX$, which implies, by Lemma~\ref{direct prod: commutativity}, that $sr=rs$. Since $r$ is an arbitrary element of $S$, it follows that $s\in\centre{S}$.
	
	Additionally, we have
	\begin{align*}
		S \text{ is commutative}&\iff S=\centre{S}\\
		&\iff \prod_{i=1}^n S_i=\prod_{i=1}^n \centre{S_i}\\
		&\iff S_i=\centre{S_i} \text{ for all } i\in \setX\\
		&\iff S_i \text{ is commutative for all } i\in\setX.\qedhere
	\end{align*}
\end{proof}

It follows from Proposition~\ref{direct prod: Z(S), S comm <=> S_i comm} that $s\in S$ is a vertex of $\commgraph{S}$ if and only if there exists $i\in \NC$ such that $s_i\in S_i\setminus \centre{S_i}$. Furthermore, as a consequence of Proposition~\ref{direct prod: Z(S), S comm <=> S_i comm}, we will assume for the reminder of the section that there exists $i\in \setX$ such that $S_i$ is not commutative, that is, we will assume that $\NC\neq\emptyset$. This way we guarantee that $S$ is not commutative and, consequently, that $\commgraph{S}$ is defined.

The first property of $\commgraph{S}$ that we investigate is its diameter (and in which situations $\commgraph{S}$ is connected). Theorem~\ref{direct prod: diameter} shows that the commutative semigroups $S_i$ ($i\in C$) do not interfere with the connectedness/diameter of $\commgraph{S}$. Furthermore, we will see that it is possible for $\commgraph{S}$ to be connected even when $\commgraph{S_i}$ is not connected for all $i\in\NC$.

\begin{lemma}\label{direct prod: lemma diameter}
	Suppose that $\NC\neq\emptyset$. Let $i\in\NC$ be such that $\NC=\set{i}$ or $\centre{S_i}=\emptyset$. If $\commgraph{S}$ is connected, then $\commgraph{S_i}$ is also connected. Furthermore, $\diam{\commgraph{S_i}}\leqslant\diam{\commgraph{S}}$.
\end{lemma}

\begin{proof}
	Let $i\in\NC$ be such that $\NC=\set{i}$ or $\centre{S_i}=\emptyset$. Suppose that $\commgraph{S}$ is connected. Let $x,y\in S_i\setminus\centre{S_i}$. For each $j\in \setX\setminus\set{i}$, let $x_j\in S_j$. Let $s,t\in S$ be such that for all $j\in \setX$
	\begin{displaymath}
		s_j=\begin{cases}
			x& \text{if } j=i,\\
			x_j& \text{if } j\neq i;
		\end{cases}\qquad
		t_j=\begin{cases}
			y& \text{if } j=i,\\
			x_j& \text{if } j\neq i.
		\end{cases}
	\end{displaymath}
	It follows from the fact that $x,y\in S_i\setminus\centre{S_i}$ and Proposition~\ref{direct prod: Z(S), S comm <=> S_i comm} that $s,t\in S \setminus\centre{S}$. Since $\commgraph{S}$ is connected, there is a path from $s$ to $t$ in $\commgraph{S}$. Let $s=s^{(1)}-s^{(2)}-\cdots-s^{(m)}=t$ be a path of minimum length, so that $m-1=\dist{\commgraph{S}}{s}{t}$.
	
	We begin by verifying that $s^{(1)}_i,\ldots,s^{(m)}_i\in S_i\setminus\centre{S_i}$. We have $\centre{S_i}=\emptyset$ or $\NC=\set{i}$. If $\centre{S_i}=\emptyset$, then $s^{(1)}_i,\ldots,s^{(m)}_i\in S_i\setminus\centre{S_i}$ because $s^{(1)}_i,\ldots,s^{(m)}_i\in S_i$. If $\NC=\set{i}$, then due to the fact that $s^{(1)},\ldots,s^{(m)}\in S\setminus\centre{S}$ and $s^{(1)}_j,\ldots,s^{(m)}_j\in S_j=\centre{S_j}$ for all $j\in C=\setX\setminus\NC=\setX\setminus\set{i}$, and by Proposition~\ref{direct prod: Z(S), S comm <=> S_i comm}, we have that $s^{(1)}_i,\ldots,s^{(m)}_i\in S_i\setminus\centre{S_i}$. 
	
	In addition, we have $s^{(k)}_is^{(k+1)}_i=s^{(k+1)}_is^{(k)}_i$ for all $k\in\X{m-1}$ because $s^{(k)}s^{(k+1)}=s^{(k+1)}s^{(k)}$ for all $k\in\X{m-1}$ and by Lemma~\ref{direct prod: commutativity}. Thus $x=s_i=s^{(1)}_i\sim s^{(2)}_i\sim\cdots\sim s^{(m)}_i=t_i=y$ (in $\commgraph{S_i}$), which implies that there exists a path from $x$ to $y$ in $\commgraph{S_i}$ and $\dist{\commgraph{S_i}}{x}{y}\leqslant m-1=\dist{\commgraph{S}}{s}{t}\leqslant \diam{\commgraph{S}}$. Since $x$ and $y$ are arbitrary elements of $S_i\setminus\centre{S_i}$, then this means that $\commgraph{S_i}$ is connected and
	\begin{displaymath}
		\diam{\commgraph{S_i}}=\max\gset{\dist{\commgraph{S_i}}{x}{y}}{x,y\in S_i\setminus\centre{S_i}}\leqslant\diam{\commgraph{S}}.\qedhere
	\end{displaymath}
\end{proof}

\begin{theorem}\label{direct prod: diameter}
	Suppose that $\NC\neq\emptyset$.
	\begin{enumerate}
		\item Suppose that $\NC=\set{i}$. Then $\commgraph{S}$ is connected if and only if $\commgraph{S_i}$ is connected, in which case we have $\diam{\commgraph{S}}=\diam{\commgraph{S_i}}$.
		
		\item Suppose that $\abs{\NC}\geqslant 2$. Then $\commgraph{S}$ is connected if and only if for all $i\in \NC$ we have $\centre{S_i}\neq\emptyset$ or $\commgraph{S_i}$ is connected. In this case we have:
		
		\begin{enumerate}
			\item If $\centre{S_i}\neq\emptyset$ for all $i\in \NC$, then $\diam{\commgraph{S}}\in\set{2,3}$. Moreover, $\diam{\commgraph{S}}=2$ if and only if there exists $j\in \NC$ such that $\diam{\commgraph{S_j}}=2$.
			
			\item If there exists $j\in \NC$ such that $\centre{S_j}=\emptyset$, then $\diam{\commgraph{S}}=\max\gset{\diam{\commgraph{S_i}}}{i\in \NC \text{ and } \centre{S_i}=\emptyset}$.
		\end{enumerate}
	\end{enumerate}
\end{theorem}

We observe that in 2.a) we are not excluding the possibility of the existence of $i\in \NC$ such that $\diam{\commgraph{S_i}}=\infty$, that is, such that $\commgraph{S_i}$ is not connected.

\begin{proof}
	\textbf{\textit{Part 1.}} Suppose that $\commgraph{S}$ is connected. Then, since $\NC=\set{i}$, Lemma~\ref{direct prod: lemma diameter} guarantees that $\commgraph{S_i}$ is connected and $\diam{\commgraph{S_i}}\leqslant\diam{\commgraph{S}}$.
	
	Now suppose that $\commgraph{S_i}$ is connected. Let $s,t\in S \setminus\centre{S}$. We have that $s_j,t_j\in S_j=\centre{S_j}$ for all $j\in C=\setX\setminus\set{i}$. Then, since $s,t\in S\setminus\centre{S}$, and as a result of Proposition~\ref{direct prod: Z(S), S comm <=> S_i comm}, we must have $s_i,t_i\in S_i\setminus\centre{S_i}$. Consequently, there exists a path from $s_i$ to $t_i$ in $\commgraph{S_i}$. Let $s_i=x_1-x_2-\cdots-x_m=t_i$ be a path from $s_i$ to $t_i$ in $\commgraph{S_i}$ such that $m-1=\dist{\commgraph{S_i}}{s_i}{t_i}$. For each $k\in\X{m}$ let $s^{(k)}\in S$ be such that
	\begin{displaymath}
		s^{(k)}_j=\begin{cases}
			x_k& \text{if } j=i,\\
			s_j& \text{if } k\neq m \text{ and } j\neq i,\\
			t_j& \text{if } k=m \text{ and } j\neq i
		\end{cases}
	\end{displaymath}
	for all $j\in \setX$. As a consequence of $x_1,\ldots,x_m\in S_i\setminus\centre{S_i}$, we have $s^{(1)},\ldots,s^{(m)}\in S\setminus\centre{S}$ (by Proposition~\ref{direct prod: Z(S), S comm <=> S_i comm}). We also have $x_kx_{k+1}=x_{k+1}x_{k}$ for all $k\in\X{m-1}$ and $s_jt_j=t_js_j$ for all $j\in C=\setX\setminus\set{i}$ (because $S_j$ is commutative for all $j\in C$). Thus, by Lemma~\ref{direct prod: commutativity}, we have that $s^{(k)}s^{(k+1)}=s^{(k+1)}s^{(k)}$ for all $k\in\X{m-1}$ and, consequently, $s=s^{(1)}\sim s^{(2)}\sim \cdots\sim s^{(m)}=t$ (in $\commgraph{S}$). This means that there is a path from $s$ to $t$ in $\commgraph{S}$ and $\dist{\commgraph{S}}{s}{t}\leqslant m-1=\dist{\commgraph{S_i}}{s_i}{t_i}\leqslant\diam{\commgraph{S_i}}$. Since $s$ and $t$ are arbitrary elements of $S\setminus\centre{S}$, then we have that $\commgraph{S}$ is connected and, additionally,
	\begin{displaymath}
		\diam{\commgraph{S}}=\max\gset{\dist{\commgraph{S}}{s}{t}}{s,t\in S\setminus\centre{S}}\leqslant \diam{\commgraph{S_i}}.
	\end{displaymath}
	
	\medskip
	
	\textbf{\textit{Part 2.}} First we are going to prove the direct implication of 2. Suppose that $\commgraph{S}$ is connected. Let $i\in \NC$ and assume that $\centre{S_i}=\emptyset$. Then we have that $\commgraph{S_i}$ is connected and $\diam{\commgraph{S_i}}\leqslant\diam{\commgraph{G}}$ (as a consequence of Lemma~\ref{direct prod: lemma diameter}).
	
	Now suppose that for all $i\in \NC$ we have $\centre{S_i}\neq\emptyset$ or $\commgraph{S_i}$ is connected. We want to prove that $\commgraph{S}$ is connected. We consider two cases: in the first one we will assume that $\centre{S_i}\neq\emptyset$ for all $i\in\NC$; and in the second one we will assume that there exists $i\in \NC$ such that $\centre{S_i}=\emptyset$.
	
	\smallskip
	
	\textbf{Case 1:} Suppose that $\centre{S_i}\neq\emptyset$ for all $i\in \NC$. We notice that we also have $\centre{S_i}\neq\emptyset$ for all $i\in C$ because $S_i$ is commutative for all $i\in C$. We are going to prove that, if $\diam{\commgraph{S_j}}=2$ for some $j\in \NC$, then $\diam{\commgraph{G}}=2$; and if $\diam{\commgraph{S_i}}\neq 2$ for all $i\in \NC$, then $\diam{\commgraph{G}}=3$.
	
	\smallskip
	
	\textit{Sub-case 1:} Suppose that there exists $j\in \NC$ such that $\diam{\commgraph{S_j}}=2$. Let $s,t\in S\setminus\centre{S}$. Then there exists $x_j\in S_j\setminus\centre{S_j}$ such that $s_j\sim x_j\sim t_j$ (in $\commgraph{S_j}$). For each $i\in \setX\setminus\set{j}$ let $z_i\in \centre{S_i}$. (We observe that $\centre{S_i}\neq\emptyset$ for all $i\in\setX$.) We define $r\in S$ as being the element such that for all $i\in \setX$
	\begin{displaymath}
		r_i=\begin{cases}
			x_j& \text{if } i=j,\\
			z_i& \text{if } i\neq j.
		\end{cases}
	\end{displaymath}
	As a result of Proposition~\ref{direct prod: Z(S), S comm <=> S_i comm}, and the fact that $x_j\in S_j\setminus\centre{S_j}$, we have that $r\in S\setminus\centre{S}$. Moreover, since $s_j\sim x_j\sim t_j$ (in $\commgraph{S_j}$) and $z_i\in\centre{S_i}$ for all $i\in \setX\setminus\set{j}$, Lemma~\ref{direct prod: commutativity} guarantees that $s\sim r\sim t$ (in $\commgraph{S}$). Thus there is a path from $s$ to $t$ in $\commgraph{S}$ and we have $\dist{\commgraph{S}}{s}{t}\leqslant 2$. Since $s$ and $t$ are arbitrary elements of $S\setminus\centre{S}$, then $\commgraph{S}$ is connected and
	\begin{displaymath}
		\diam{\commgraph{S}}\leqslant\max\gset{\dist{\commgraph{S}}{s}{t}}{s,t\in S\setminus\centre{S}}\leqslant 2.
	\end{displaymath}
	The result follows from the fact that $\diam{\commgraph{S}}\geqslant 2$, which concludes the proof of sub-case 1.
	
	\smallskip
	
	\textit{Sub-case 2:} Now suppose that $\diam{\commgraph{S_i}}\neq 2$ for all $i\in \NC$. Let $s,t\in S\setminus\centre{S}$. Then, by Proposition~\ref{direct prod: Z(S), S comm <=> S_i comm}, there exists $j\in \NC$ such that $s_j\in S_j\setminus\centre{S_j}$. Furthermore, Proposition~\ref{direct prod: Z(S), S comm <=> S_i comm} also guarantees that $J=\gset{i\in \NC}{t_i\in S_i\setminus \centre{S_i}}\neq\emptyset$. We consider two sub-sub-cases: $J=\set{j}$ and $J\neq\set{j}$.
	
	\smallskip
	
	\textsc{Sub-sub-case 1:} Assume that $J=\set{j}$. Since $\abs{\NC}\geqslant 2$, then there exists $k\in \NC$ such that $k\neq j$. We have $t_k\in\centre{S_k}$. Let $x_k\in S_k\setminus\centre{S_k}$ and for each $i\in \NC\setminus\set{k}$ let $z_i\in\centre{S_i}$. Let $r^{(1)},r^{(2)}\in S$ be such that for all $i\in \setX$
	\begin{displaymath}
		r^{(1)}_i=\begin{cases}
			t_k& \text{if } i=k,\\
			s_i& \text{if } i\neq k;
		\end{cases}\qquad
		r^{(2)}_i=\begin{cases}
			x_k& \text{if } i=k,\\
			z_i& \text{if } i\in \NC\setminus\set{k},\\
			s_i& \text{if } i\in C.
		\end{cases}
	\end{displaymath}
	Due to the fact that $s_j\in S_j\setminus\centre{S_j}$ and $x_k\in S_k\setminus\centre{S_k}$ (and as a result of Proposition~\ref{direct prod: Z(S), S comm <=> S_i comm}), we have $r^{(1)},r^{(2)}\in S\setminus\centre{S}$. Moreover, Lemma~\ref{direct prod: commutativity} ensures that we have $sr^{(1)}=r^{(1)}s$ (because $t_k\in\centre{S_k}$), $r^{(1)}r^{(2)}=r^{(2)}r^{(1)}$ (because $t_k\in\centre{S_k}$ and $z_i\in\centre{S_i}$ for all $i\in \NC\setminus\set{k}$) and $r^{(2)}t=tr^{(2)}$ (because $t_k\in\centre{S_k}$, $z_i\in\centre{S_i}$ for all $i\in \NC\setminus\set{k}$ and $S_i$ is commutative for all $i\in C$). Thus $s\sim r^{(1)}\sim r^{(2)}\sim t$, which implies that there is a path from $s$ to $t$ in $\commgraph{S}$ and $\dist{\commgraph{S}}{s}{t}\leqslant 3$.
	
	\smallskip
	
	\textsc{Sub-sub-case 2:} Assume that $J\neq\set{j}$. Let $k\in J\setminus\set{j}$. Then $t_k\in S_k\setminus\centre{S_k}$. For each $i\in \NC$ we choose $z_i\in\centre{S_i}$. (We recall that $\centre{S_i}\neq\emptyset$ for all $i\in\NC$.) Let $r^{(1)},r^{(2)}\in S$ be such that for all $i\in \setX$
	\begin{displaymath}
		r^{(1)}_i=\begin{cases}
			z_k& \text{if } i=k,\\
			s_i& \text{if } i\neq k;
		\end{cases}\qquad
		r^{(2)}_i=\begin{cases}
			t_k& \text{if } i=k,\\
			z_i& \text{if } i\in \NC\setminus\set{k},\\
			s_i& \text{if } i\in C.
		\end{cases}
	\end{displaymath}
	As a result of Proposition~\ref{direct prod: Z(S), S comm <=> S_i comm}, and the fact that $s_j\in S_j\setminus\centre{S_j}$ and $t_k\in S_k\setminus\centre{S_k}$, we have $r^{(1)},r^{(2)}\in S\setminus\centre{S}$. Additionally, by Lemma~\ref{direct prod: commutativity}, we have $sr^{(1)}=r^{(1)}s$ (because $z_k\in \centre{S_k}$), $r^{(1)}r^{(2)}=r^{(2)}r^{(1)}$ (because $z_i\in \centre{S_i}$ for all $i\in \NC$) and $r^{(2)}t=tr^{(2)}$ (because $z_i\in \centre{S_i}$ for all $i\in \NC\setminus\set{k}$ and $S_i$ is commutative for all $i\in C$). Thus $s\sim r^{(1)}\sim r^{(2)}\sim t$ and, consequently, there is a path from $s$ to $t$ in $\commgraph{S}$ and $\dist{\commgraph{S}}{s}{t}\leqslant 3$.
	
	\smallskip
	
	In both sub-sub-cases we concluded that $\dist{\commgraph{S}}{s}{t}\leqslant 3$. Since $s$ and $t$ are arbitrary elements of $S\setminus\centre{S}$, then we have that $\commgraph{S}$ is connected and
	\begin{displaymath}
		\diam{\commgraph{S}}\leqslant\max\gset{\dist{\commgraph{S}}{s}{t}}{s,t\in S\setminus\centre{S}}\leqslant 3.
	\end{displaymath}
	
	Now we are going to see that $\diam{\commgraph{S}}\geqslant 3$. Since $\diam{\commgraph{S_i}}>2$ for all $i \in \NC$, then for each $i\in \NC$ there exist $x_i,y_i\in S_i\setminus\centre{S_i}$ such that $\dist{\commgraph{S_i}}{x_i}{y_i}>2$. For each $i\in C$ we select $z_i\in S_i$. Let $s,t\in S$ be such that
	\begin{displaymath}
		s_i=\begin{cases}
			x_i& \text{if } i\in \NC,\\
			z_i& \text{if } i\in C;
		\end{cases}\qquad
		t_i=\begin{cases}
			y_i& \text{if } i\in \NC,\\
			z_i& \text{if } i\in C
		\end{cases}
	\end{displaymath}
	for all $i\in \setX$. Let $r\in S$ be such that $sr=rs$ and $rt=tr$. It follows from Lemma~\ref{direct prod: commutativity} that $x_ir_i=s_ir_i=r_is_i=r_ix_i$ and $r_iy_i=r_it_i=t_ir_i=y_ir_i$ for all $i\in \NC$. Since $\dist{\commgraph{S_i}}{x_i}{y_i}>2$ for all $i\in \NC$, then we must have $r_i\in\centre{S_i}$ for all $i\in \NC$. In addition, we also have $r_i=z_i\in S_i=\centre{S_i}$ for all $i\in C$. Thus, by Proposition~\ref{direct prod: Z(S), S comm <=> S_i comm}, $r\in\centre{S}$, which implies that $\diam{\commgraph{S}}\geqslant\dist{\commgraph{S}}{s}{t}>2$. This concludes sub-case 2 and thus case 1.
	
	\smallskip
	
	\textbf{Case 2:} Suppose that $I=\gset{i\in \NC}{\centre{S_i}=\emptyset}\neq\emptyset$. Then $\commgraph{S_i}$ is connected for all $i\in I$ (and we have $\diam{\commgraph{S_i}}\geqslant 2$ for all $i\in I$). Let $s,t\in S\setminus\centre{S}$. Our aim is to prove that there exists a path from $s$ to $t$ in $\commgraph{S}$.
	
	\smallskip
	
	\textit{Sub-case 1:} Assume that $s_it_i=t_is_i$ for all $i\in I$. For each $i\in \setX\setminus I=\parens{\NC\setminus I} \cup C$ we choose $z_i\in\centre{S_i}$. (We observe that it follows from the definition of $I$ that $\centre{S_i}\neq\emptyset$ for all $i\in \NC\setminus I$ and it follows from the definition of $C$ that $\centre{S_i}\neq\emptyset$ for all $i\in C$.) We define $r\in S$ as being the element such that for all $i\in \setX$
	\begin{displaymath}
		r_i=\begin{cases}
			s_i& \text{if } i\in I,\\
			z_i& \text{if } i\in \setX\setminus I.
		\end{cases}
	\end{displaymath}
	Since $\centre{S_i}=\emptyset$ for all $i\in I$, then Proposition~\ref{direct prod: Z(S), S comm <=> S_i comm} guarantees that $r\in S\setminus\centre{S}$. Additionally, we have $sr=rs$ and $rt=tr$ because $s_it_i=t_is_i$ for all $i\in I$ and $r_i=z_i\in\centre{S_i}$ for all $i\in \setX \setminus I$, and by Lemma~\ref{direct prod: commutativity}. Thus $s\sim r\sim t$ (in $\commgraph{S}$), which implies that there exists a path from $s$ to $t$ in $\commgraph{S}$ and
	\begin{displaymath}
		\dist{\commgraph{S}}{s}{t}\leqslant 2\leqslant \max\gset{\diam{\commgraph{S_i}}}{i\in I}.
	\end{displaymath}
	
	\smallskip
	
	\textit{Sub-case 2:} Now assume that there exists $j\in I$ such that $s_jt_j\neq t_js_j$. Then we have $\dist{\commgraph{S_j}}{s_j}{t_j}\geqslant 2$. For each $i\in I$ there exists a path from $s_i$ to $t_i$ in $\commgraph{S_i}$. For each $i\in I$ let $s_i=s_{i1}-s_{i2}-\cdots-s_{im_i}=t_i$ be a path from $s_i$ to $t_i$ such that $m_i-1=\dist{\commgraph{S_i}}{s_i}{t_i}$. Let $m=\max\gset{m_i}{i\in I}$. (We observe that we have $m\geqslant m_j=\dist{\commgraph{S_j}}{s_j}{t_j}+1\geqslant 3$.) We choose $z_i\in \centre{S_i}$ for all $i \in \NC\setminus I$. For each $k\in\X{m}$ we define $s^{(k)}\in S$ as the element such that
	\begin{displaymath}
		s^{(k)}_i=\begin{cases}
			s_{ik}& \text{if } k<m_i \text{ and } i\in I,\\
			s_{im_i}& \text{if } m_i\leqslant k\leqslant m \text{ and } i\in I,\\
			s_i& \text{if } k=1 \text{ and } i\in \NC\setminus I,\\
			z_i& \text{if } k=2 \text{ and } i\in \NC\setminus I,\\
			t_i& \text{if } 2< k\leqslant m \text{ and } i\in \NC\setminus I,\\
			s_i& \text{if } k\neq m \text{ and } i\in C,\\
			t_i& \text{if } k=m \text{ and } i\in C
		\end{cases}
	\end{displaymath}
	for all $i\in \setX$. For all $k\in\X{m-1}$ and $i\in C$ we have $s^{(k)}_is^{(k+1)}_i=s^{(k+1)}_is^{(k)}_i$ because $S_i$ is commutative for all $i\in C$. We also have $s^{(1)}_is^{(2)}_i=s_iz_i=z_is_i=s^{(2)}_is^{(1)}_i$ and $s^{(2)}_is^{(3)}_i=z_it_i=t_iz_i=s^{(3)}_is^{(2)}_i$ for all $i\in \NC\setminus I$ --- because $s^{(2)}_i=z_i\in \centre{S_i}$ for all $i\in \NC\setminus I$ --- and we have $s^{(k)}_is^{(k+1)}_i=t_it_i=s^{(k+1)}_is^{(k)}_i$ for all $k\in\set{3,\ldots,m-1}$ and $i\in \NC\setminus I$. Finally, for all $i\in I$ and $k\in\X{m_i-1}$ we have $s^{(k)}_is^{(k+1)}_i=s_{ik}s_{i\parens{k+1}}=s_{i\parens{k+1}}s_{ik}=s^{(k+1)}_is^{(k)}_i$ --- because for all $i\in I$ and $k\in\X{m_i-1}$ we have $s_{ik}\sim s_{i(k+1)}$ (in $\commgraph{S_i}$) --- and for all $i\in I$ and $k\in\set{m_i,\ldots,m-1}$ we have $s^{(k)}_is^{(k+1)}_i=s_{im_i}s_{im_i}=s^{(k+1)}_is^{(k)}_i$. Thus Lemma~\ref{direct prod: commutativity} guarantees that $s^{(k)}s^{(k+1)}=s^{(k+1)}s^{(k)}$ for all $k\in\X{m-1}$ and, consequently, we have $s=s^{(1)}\sim s^{(2)}\sim\cdots\sim s^{(m)}=t$ (in $\commgraph{S}$). Therefore there exists a path from $s$ to $t$ in $\commgraph{S}$ and
	\begin{align*}
		\dist{\commgraph{S}}{s}{t}&\leqslant m-1\\
		&=\max\gset{m_i-1}{i\in I}\\
		&=\max\gset{\dist{\commgraph{S_i}}{s_i}{t_i}}{i\in I}\\
		&\leqslant\max\gset{\diam{\commgraph{S_i}}}{i\in I}.
	\end{align*}
	
	\smallskip
	
	Since $s$ and $t$ are arbitrary elements of $S\setminus\centre{S}$, we just proved that $\commgraph{S}$ is connected and
	\begin{displaymath}
		\diam{\commgraph{S}}\leqslant\max\gset{\dist{\commgraph{S}}{s}{t}}{s,t\in S\setminus\centre{S}}\leqslant \max\gset{\diam{\commgraph{S_i}}}{i\in I}.
	\end{displaymath}
	Moreover, as a consequence of Lemma~\ref{direct prod: lemma diameter} we have that, when $\commgraph{S}$ is connected, then $\diam{\commgraph{S_i}}\leqslant \diam{\commgraph{S}}$ for all $i\in I$. Hence we have $\max\gset{\diam{\commgraph{S_i}}}{i\in I}\leqslant\diam{\commgraph{S}}$ and, consequently, $\diam{\commgraph{S}}=\max\gset{\diam{\commgraph{S_i}}}{i\in I}$.
\end{proof}

Arvind et al.~\cite{Graphs_arise_as_commuting_graphs_groups} showed that, when $G_1$ and $G_2$ are groups, then $\extendedcommgraph{G_1\times G_2}$ is isomorphic to $\extendedcommgraph{G_1}\boxtimes\extendedcommgraph{G_2}$. The following result states that this is also true when, instead of two groups, we consider two semigroups and, more generally, when we think about the direct product of $n$ semigroups.

\begin{proposition}\label{direct prod: characterization G*(S) with strong prod}
	We have that $\extendedcommgraph{S}$ is isomorphic to $\strongprod_{i=1}^n\extendedcommgraph{S_i}$.
\end{proposition}

\begin{proof}
	For each $i \in \setX$ the set of vertices of $\extendedcommgraph{S_i}$ is $S_i$. Hence the set of vertices of $\strongprod_{i=1}^n\extendedcommgraph{S_i}$ is $\prod_{i=1}^n S_i$, which is set of vertices of $\extendedcommgraph{S}$. Additionally, we have 
	\begin{align*}
		& s \text{ and } t \text{ are adjacent vertices in } \extendedcommgraph{S}\\
		\iff{} & s\neq t \text{ and } st=ts\\
		\iff{} & s\neq t \text{ and } s_it_i=t_is_i \text{ for all } i\in \setX & \bracks{\text{by Lemma~\ref{direct prod: commutativity}}}\\
		\iff{} & s\neq t \text{ and } s_i \sim t_i \text{ for all } i\in \setX\\
		\iff{} & s \text{ and } t \text{ are adjacent vertices in } \strongprod_{i=1}^n\extendedcommgraph{S_i}.&&\qedhere
	\end{align*}
\end{proof}

In the following two theorems we use Proposition~\ref{direct prod: characterization G*(S) with strong prod} to determine the clique number (Theorem~\ref{direct prod: clique number}) and an upper bound for the chromatic number (Theorem~\ref{direct prod: chromatic number}) of $\commgraph{S}$.

\begin{theorem}\label{direct prod: clique number}
	Suppose that $\NC\neq\emptyset$. We have
	\begin{displaymath}
		\cliquenumber{\commgraph{S}}=\parens[\bigg]{\,\prod_{i\in C}\abs{S_i}}\parens[\bigg]{\,\prod_{i\in \NC}{\parens[\big]{\abs{\centre{S_i}}+\cliquenumber{\commgraph{S_i}}}}} -\prod_{i=1}^n{\abs{\centre{S_i}}}.
	\end{displaymath}
\end{theorem}

\begin{proof}
	
	
	Since $\NC\neq\emptyset$, and by Proposition~\ref{direct prod: Z(S), S comm <=> S_i comm}, we have that $S$ is not commutative. Hence, by Lemma~\ref{preli: commgraph and extended commgraph}, $\extendedcommgraph{S}$ is isomorphic to $K_{\centre{S}}\graphjointwo\commgraph{S}$. Thus
	\begin{align*}
		\cliquenumber{\commgraph{S}} & = \cliquenumber{\commgraph{S}} + \cliquenumber{\extendedcommgraph{S}} - \cliquenumber{\extendedcommgraph{S}}\\
		& =\cliquenumber{\commgraph{S}}+ \cliquenumber{\extendedcommgraph{S}}-\cliquenumber{K_{\centre{S}}\graphjointwo\commgraph{S}}\\
		& =\cliquenumber{\commgraph{S}}+ \cliquenumber{\extendedcommgraph{S}}-\parens[\big]{\cliquenumber{K_{\centre{S}}}+\cliquenumber{\commgraph{S}}} & \bracks{\text{by Lemma~\ref{preli: graph join clique/chromatic numbers}}}\\
		& =\cliquenumber{\extendedcommgraph{S}}-\cliquenumber{K_{\centre{S}}}\\
		& =\cliquenumber{\extendedcommgraph{S}}-\abs{\centre{S}}\\
		& =\cliquenumber{\extendedcommgraph{S}}-\abs[\Big]{\prod_{i=1}^n\centre{S_i}}& \bracks{\text{by Proposition~\ref{direct prod: Z(S), S comm <=> S_i comm}}}\\
		& =\cliquenumber{\extendedcommgraph{S}}-\prod_{i=1}^n\abs{\centre{S_i}}.
	\end{align*}
	
	The only thing left to do is to determine $\cliquenumber{\extendedcommgraph{S}}$. We have
	\begin{align*}
		& \cliquenumber{\extendedcommgraph{S}}\\
		={} & \cliquenumber[\Big]{\,\strongprod_{i=1}^n\extendedcommgraph{S_i}} & \bracks{\text{by Proposition~\ref{direct prod: characterization G*(S) with strong prod}}}\\
		={} & \cliquenumber[\bigg]{\parens[\Big]{\parens[\big]{\extendedcommgraph{S_1}\boxtimes\extendedcommgraph{S_2}}\boxtimes\extendedcommgraph{S_3}}\cdots\boxtimes\extendedcommgraph{S_n}}\\
		={} & \parens[\Big]{\parens[\big]{\cliquenumber{\extendedcommgraph{S_1}}\cdot\cliquenumber{\extendedcommgraph{S_2}}}\cdot\cliquenumber{\extendedcommgraph{S_3}}}\cdots\cdot\cliquenumber{\extendedcommgraph{S_n}} & \smash{\begin{minipage}[c]{30mm}\raggedleft [by iterated use of Lemma~\ref{preli: strong prod clique/chromatic numbers}]\end{minipage}}\\
		={} & \prod_{i=1}^n \cliquenumber{\extendedcommgraph{S_i}}\\
		={} & \parens[\bigg]{\,\prod_{i\in C} \cliquenumber{\extendedcommgraph{S_i}}} \parens[\bigg]{\,\prod_{i\in \NC} \cliquenumber{\extendedcommgraph{S_i}}}\\
		={} & \parens[\bigg]{\,\prod_{i\in C} \cliquenumber{ K_{\abs{S_i}}}}\parens[\bigg]{\,\prod_{i\in \NC} \cliquenumber{K_{\abs{\centre{S_i}}}\graphjointwo\commgraph{S_i} }} & \bracks{\text{by Lemma~\ref{preli: commgraph and extended commgraph}}}\\
		={}& \parens[\bigg]{\,\prod_{i\in C} \cliquenumber{ K_{\abs{S_i}}}}\parens[\bigg]{\,\prod_{i\in \NC} \cliquenumber{K_{\abs{\centre{S_i}}}}+\cliquenumber{\commgraph{S_i}}} & \bracks{\text{by Lemma~\ref{preli: graph join clique/chromatic numbers}}}\\
		={}& \parens[\bigg]{\,\prod_{i\in C} \abs{S_i}}\parens[\bigg]{\,\prod_{i\in \NC} \abs{\centre{S_i}}+\cliquenumber{\commgraph{S_i}}}.
	\end{align*}
	
	Therefore
	\begin{displaymath}
		\cliquenumber{\commgraph{S}}=\parens[\bigg]{\,\prod_{i\in C} \abs{S_i}}\parens[\bigg]{\,\prod_{i\in \NC} \abs{\centre{S_i}}+\cliquenumber{\commgraph{S_i}}}-\prod_{i=1}^n\abs{\centre{S_i}}.\qedhere
	\end{displaymath}
\end{proof}

We observe that Theorem~\ref{direct prod: clique number} provides a lower bound for $\chromaticnumber{\commgraph{S}}$. In the next Theorem we present an upper bound for $\chromaticnumber{\commgraph{S}}$.

\begin{theorem}\label{direct prod: chromatic number}
	Suppose that $\NC\ne\emptyset$. We have
	\begin{displaymath}
		\chromaticnumber{\commgraph{S}}\leqslant\parens[\bigg]{\,\prod_{i\in C}\abs{S_i}}\parens[\bigg]{\,\prod_{i\in \NC}{\parens[\big]{\abs{\centre{S_i}}+\chromaticnumber{\commgraph{S_i}}}}} -\prod_{i=1}^n{\abs{\centre{S_i}}}.
	\end{displaymath}
\end{theorem}

\begin{proof}
	As a consequence of $\NC\neq\emptyset$, and by Proposition~\ref{direct prod: Z(S), S comm <=> S_i comm}, we have that $S$ is not commutative. Then it follows from Lemma~\ref{preli: commgraph and extended commgraph} that $\extendedcommgraph{S}$ is isomorphic to $K_{\abs{\centre{S}}}\graphjointwo\commgraph{S}$ and, consequently, we have
	\begin{align*}
		\chromaticnumber{\extendedcommgraph{S}}&=\chromaticnumber{K_{\abs{\centre{S}}}\graphjointwo\commgraph{S}}\\
		&=\chromaticnumber{K_{\abs{\centre{S}}}}+\chromaticnumber{\commgraph{S}}& \bracks{\text{by Lemma~\ref{preli: graph join clique/chromatic numbers}}}\\
		&=\abs{\centre{S}}+\chromaticnumber{\commgraph{S}}\\
		&=\abs[\Big]{\prod_{i=1}^n\centre{S_i}} + \chromaticnumber{\commgraph{S}}& \bracks{\text{by Proposition~\ref{direct prod: Z(S), S comm <=> S_i comm}}}\\
		&=\prod_{i=1}^n\abs{\centre{S_i}} + \chromaticnumber{\commgraph{S}}.
	\end{align*}
	
	Furthermore, we have
	\begin{align*}
		&\chromaticnumber{\extendedcommgraph{S}}\\
		={} & \chromaticnumber[\Big]{\,\strongprod_{i=1}^n\extendedcommgraph{S_i}} & \bracks{\text{by Proposition~\ref{direct prod: characterization G*(S) with strong prod}}}\\
		={} & \chromaticnumber[\bigg]{\parens[\Big]{\parens[\big]{\extendedcommgraph{S_1}\boxtimes\extendedcommgraph{S_2}}\boxtimes\extendedcommgraph{S_3}}\cdots\boxtimes\extendedcommgraph{S_n}}\\
		\leqslant{} & \parens[\Big]{\parens[\big]{\chromaticnumber{\extendedcommgraph{S_1}}\cdot\chromaticnumber{\extendedcommgraph{S_2}}}\cdot\chromaticnumber{\extendedcommgraph{S_3}}}\cdots\cdot\chromaticnumber{\extendedcommgraph{S_n}} & \smash{\begin{minipage}[c]{30mm}\raggedleft [by iterated use of Lemma~\ref{preli: strong prod clique/chromatic numbers}]\end{minipage}}\\
		={}& \prod_{i=1}^n\chromaticnumber{\extendedcommgraph{S_i}}\\
		={}& \parens[\bigg]{\,\prod_{i\in C}\chromaticnumber{\extendedcommgraph{S_i}}}\parens[\bigg]{\,\prod_{i\in \NC}\chromaticnumber{\extendedcommgraph{S_i}}}\\
		={}&\parens[\bigg]{\,\prod_{i\in C}\chromaticnumber{K_{\abs{S_i}}}}\parens[\bigg]{\,\prod_{i\in \NC}\chromaticnumber[\big]{K_{\abs{\centre{S_i}}}\graphjointwo\commgraph{S_i}}}& \bracks{\text{by Lemma~\ref{preli: commgraph and extended commgraph}}}\\
		={}&\parens[\bigg]{\,\prod_{i\in C}\chromaticnumber{K_{\abs{S_i}}}}\parens[\bigg]{\,\prod_{i\in \NC}\parens[\big]{\chromaticnumber{K_{\abs{\centre{S_i}}}}+\chromaticnumber{\commgraph{S_i}}}}& \bracks{\text{by Lemma~\ref{preli: graph join clique/chromatic numbers}}}\\
		={}&\parens[\bigg]{\,\prod_{i\in C}\abs{S_i}}\parens[\bigg]{\,\prod_{i\in \NC}\parens[\big]{\abs{\centre{S_i}}+\chromaticnumber{\commgraph{S_i}}}}.
	\end{align*}
	
	Therefore
	\begin{align*}
		\chromaticnumber{\commgraph{S}}&=\chromaticnumber{\extendedcommgraph{S}}-\prod_{i=1}^n\abs{\centre{S_i}}\\
		&\leqslant \parens[\bigg]{\,\prod_{i\in C}\abs{S_i}}\parens[\bigg]{\,\prod_{i\in \NC}\parens{\abs{\centre{S_i}}+\chromaticnumber{\commgraph{S_i}}}}-\prod_{i=1}^n\abs{\centre{S_i}}.\qedhere
	\end{align*}
\end{proof}

Theorem~\ref{direct prod: girth} characterizes the situations in which $\commgraph{S}$ contains cycles and it provides a way to determine the length of a shortest cycle in $\commgraph{G}$. Before we prove Theorem~\ref{direct prod: girth}, we establish (and prove) the following lemma, which will simplify some cases of the proof of Theorem~\ref{direct prod: girth}.

\begin{lemma}\label{direct prod: lemma girth}
	Suppose that $\NC\neq\emptyset$. Suppose that one of the following three conditions holds:
	\begin{enumerate}
		\item There exist $i\in \setX$ such that $\NC\setminus\set{i}\neq\emptyset$, and distinct $x,y,z\in S_i$ such that $x$, $y$ and $z$ commute with each other.
		
		\item There exist distinct $i,j\in \setX$ such that $\NC\setminus \set{i,j}\neq\emptyset$, distinct $x,y\in S_i$ such that $xy=yx$, and distinct $z,w\in S_j$ such that $zw=wz$.
		
		\item There exist $i\in \setX$ and $j\in\NC$, distinct $x,y\in S_i$ such that $xy=yx$, and $z\in S_j$ and $w\in S_j\setminus\centre{S_j}$ such that $zw=wz$. If $x\in S_i\setminus\centre{S_i}$ or $z\in S_j\setminus\centre{S_j}$, then $\commgraph{S}$ contains a cycle of length 3.
	\end{enumerate}
	Then $\commgraph{S}$ contains a cycle of length 3.
\end{lemma}

\begin{proof}
	\textbf{Part 1.} Suppose that there exist $i\in \setX$ with $\NC\setminus\set{i}\neq\emptyset$, and distinct $x,y,z\in S_i$ such that $x$, $y$ and $z$ commute with each other. Then there exists $j\in\NC\setminus\set{i}$. We choose an element $w\in S_j\setminus\centre{S_j}$, and for each $k\in \setX\setminus\set{i,j}$ we choose an element $z_k\in S_k$. Let $s,t,r\in S$ be such that for each $k\in \setX$
	\begin{align*}
		s_k&=\begin{cases}
			x& \text{if }k=i,\\
			w& \text{if }k=j,\\
			z_k& \text{if }k\in \setX \setminus\set{i,j};
		\end{cases}&
		t_k=\begin{cases}
			y& \text{if }k=i,\\
			w& \text{if }k=j,\\
			z_k& \text{if }k\in \setX \setminus\set{i,j};
		\end{cases}\\
		r_k&=\begin{cases}
			z& \text{if }k=i,\\
			w& \text{if }k=j,\\
			z_k& \text{if }k\in \setX \setminus\set{i,j}.
		\end{cases}
	\end{align*}
	Since $s_j=t_j=r_j=w\in S_j\setminus\centre{S_j}$, then Proposition~\ref{direct prod: Z(S), S comm <=> S_i comm} guarantees that $s,t,r\in S\setminus \centre{S}$. Furthermore, it follows from Lemma~\ref{direct prod: commutativity}, and the fact that $x$, $y$ and $z$ commute with each other, that $s$, $t$ and $r$ commute with each other and, consequently, $s-t-r-s$ is a cycle (of length 3) in $\commgraph{S}$.
	
	\medskip
	
	\textbf{Part 2.} Suppose that there exist distinct $i,j\in \setX$ with $\NC\setminus \set{i,j}\neq\emptyset$, distinct $x,y\in S_i$ such that $xy=yx$, and distinct $z,w\in S_j$ such that $zw=wz$. Let $l\in \NC\setminus\set{i,j}$. We select an element $u\in S_l\setminus\centre{S_l}$, and for each $k\in \setX\setminus\set{i,j,l}$ we select an element $z_k\in S_k$. We then use these elements to define $s,t,r\in S$ in the following way:
	\begin{align*}
		s_k&=\begin{cases}
			x& \text{if }k=i,\\
			z& \text{if }k=j,\\
			u& \text{if }k=l,\\
			z_k& \text{if }k\in \setX\setminus\set{i,j,l};
		\end{cases}&
		t_k=\begin{cases}
			y& \text{if }k=i,\\
			z& \text{if }k=j,\\
			u& \text{if }k=l,\\
			z_k& \text{if }k\in \setX \setminus\set{i,j,l};
		\end{cases}\\
		r_k&=\begin{cases}
			y& \text{if }k=i,\\
			w& \text{if }k=j,\\
			u& \text{if }k=l,\\
			z_k& \text{if }k\in \setX \setminus\set{i,j,l}
		\end{cases}
	\end{align*}
	for all $k\in \setX$. Due to the fact that $s_l=t_l=r_l=u\in S_l\setminus\centre{S_l}$, then we have $s,t,r\in S\setminus \centre{S}$ (by Proposition~\ref{direct prod: Z(S), S comm <=> S_i comm}). Furthermore, as a consequence of Lemma~\ref{direct prod: commutativity}, and the fact that $xy=yx$ and $zw=wz$, we have that $s$, $t$ and $r$ commute with each other. Therefore $s-t-r-s$ is a cycle (of length 3) in $\commgraph{S}$.
	
	\medskip
	
	\textbf{Part 3.} Suppose that there exist $i\in \setX$ and $j\in\NC$, distinct $x,y\in S_i$ such that $xy=yx$, and $z\in S_j$ and $w\in S_j\setminus\centre{S_j}$ such that $zw=wz$. Assume that $x\in S_i\setminus\centre{S_i}$ or $z\in S_j\setminus\centre{S_j}$. For each $k\in \setX\setminus\set{i,j}$ we choose $x_k\in S_k$. Let $s,t,r\in S$ be such that for all $k\in \setX$
	\begin{align*}
		s_k&=\begin{cases}
			x& \text{if }k=i,\\
			z& \text{if }k=j,\\
			x_k& \text{if }k\in \setX \setminus\set{i,j};
		\end{cases}&
		t_k=\begin{cases}
			x& \text{if }k=i,\\
			w& \text{if }k=j,\\
			x_k& \text{if }k\in \setX \setminus\set{i,j};
		\end{cases}\\
		r_k&=\begin{cases}
			y& \text{if }k=i,\\
			w& \text{if }k=j,\\
			x_k& \text{if }k\in \setX \setminus\set{i,j}.
		\end{cases}
	\end{align*}
	Since $x\in S_i\setminus\centre{S_i}$ or $z\in S_j\setminus\centre{S_j}$, then it follows from Proposition~\ref{direct prod: Z(S), S comm <=> S_i comm} that $s\in S\setminus\centre{S}$. Moreover, we have $t_j=r_j=w\in S_j\setminus\centre{S_j}$, which implies (by Proposition~\ref{direct prod: Z(S), S comm <=> S_i comm}) that $t,r\in S\setminus\centre{S_j}$. As a consequence of the fact that $xy=yx$ and $zw=wz$, and due to Lemma~\ref{direct prod: commutativity}, we have that $s$, $t$ and $r$ commute with each other. Therefore $s-t-r-s$ is a cycle in $\commgraph{S}$ (of length $3$).
\end{proof}

\begin{theorem}\label{direct prod: girth}
	Suppose that $\NC\neq\emptyset$. We have that $\commgraph{S}$ contains cycles if and only if at least one of the following conditions holds:
	\begin{enumerate}
		\item There exists $i\in C$ such that $\abs{S_i}\geqslant 3$.
		
		\item There exist distinct $i,j\in C$ such that $\abs{S_i}\geqslant 2$ and $\abs{S_j}\geqslant 2$.
		
		\item There exist $i\in C$ and $j\in \NC$ such that $\abs{S_i}\geqslant 2$ and $\commgraph{S_j}$ is not a null graph.
		
		\item $\abs{\NC}\geqslant 2$ and there exist $i\in C$ and $j\in \NC$ such that $\abs{S_i}\geqslant 2$ and $\centre{S_j}\neq\emptyset$.
		
		\item There exist distinct $i,j\in \NC$ such that either $\centre{S_i}\neq\emptyset$ or $\commgraph{S_i}$ is not a null graph and either $\centre{S_j}\neq\emptyset$ or $\commgraph{S_j}$ is not a null graph.
		
		\item $\abs{\NC}\geqslant 2$ and there exist $i\in \NC$ such that $\abs{\centre{S_i}}\geqslant 2$.
		
		\item $\abs{\NC}\geqslant 2$ and there exist $i\in \NC$ such that $\centre{S_i}\neq\emptyset$ and $\commgraph{S_i}$ is not a null graph.
		
		\item There exist $i\in \NC$ such that $\commgraph{S_i}$ contains cycles.
	\end{enumerate}
	Furthermore, if at least one of the conditions 1--7 is satisfied, then $\girth{\commgraph{S}}=3$, and if condition 8 is the only one that is satisfied, then there exists a unique $i\in \NC$ such that $\commgraph{S_i}$ contains cycles and we have $\girth{\commgraph{S}}=\girth{\commgraph{S_i}}$.
\end{theorem}

\begin{proof}
	
	We are going to divide this proof into three parts.
	
	\medskip
	
	\textbf{Part 1.} We are going to see that each one of the conditions 1--8 implies the existence of a cycle in $\commgraph{S}$.
	
	\smallskip
	
	\textit{Case 1:} Assume that condition 1 holds. Since $\abs{S_i}\geqslant 3$, then there exist distinct $x,y,z\in S_i$. In addition, $S_i$ is commutative, which implies that $x$, $y$ and $z$ commute with each other. Moreover, we have $\NC\setminus\set{i}=\NC\neq\emptyset$ because $i\in C=\setX\setminus\NC$. Therefore, condition 1 of Lemma~\ref{direct prod: lemma girth} holds and, consequently, $\commgraph{S}$ contains a cycle (of length $3$).
	
	\smallskip
	
	\textit{Case 2:} Assume that condition 2 holds. There exist distinct $x,y\in S_i$ and distinct $z,w\in S_j$. Due to the fact that $S_i$ and $S_j$ are commutative, we have $xy=yx$ and $zw=wz$. Since we also have $\NC\setminus\set{i,j}=\NC\neq\emptyset$ (because $i,j\in C=\setX\setminus\NC$), then condition 2 of Lemma~\ref{direct prod: lemma girth} is satisfied and, consequently, we can conclude that $\commgraph{S}$ contains a cycle (of length $3$).
	
	\smallskip
	
	\textit{Case 3:} Assume that condition 3 holds. It follows from the fact that $\abs{S_i}\geqslant 2$ that there exist distinct $x,y\in S_i$, and it follows from the fact that $\commgraph{S_j}$ is not a null graph that there exist $z,w\in S_j\setminus\centre{S_j}$ such that $z$ and $w$ are adjacent vertices of $\commgraph{S_j}$. Hence $xy=yx$ and $zw=wz$. Hence condition 3 of Lemma~\ref{direct prod: lemma girth} is satisfied, which implies the existence of a cycle (of length $3$) in $\commgraph{S}$.
	
	\smallskip
	
	\textit{Case 4:} Assume that condition 4 holds. Since $\abs{S_i}\geqslant 2$ and $S_i$ is commutative, then there exist distinct $x,y\in S_i$ that verify $xy=yx$. We also have $\centre{S_j}\neq\emptyset$ and $S_j\setminus\centre{S_j}\neq\emptyset$ (because $S_j$ is not commutative), which implies that there exist $z\in\centre{S_j}$ and $w\in S_j\setminus\centre{S_j}$. Hence $zw=wz$. Furthermore, $\abs{\NC\setminus\set{i,j}}=\abs{\NC\setminus\set{j}}=\abs{\NC}-1\geqslant 1>0$ (because $i\in C=\setX\setminus\NC$, $j\in\NC$ and $\abs{\NC}\geqslant 2$). Thus condition 2 of Lemma~\ref{direct prod: lemma girth} holds and we have $\commgraph{S}$ contains a cycle (of length $3$).
	
	\smallskip
	
	\textit{Case 5:} Assume that condition 5 holds. First we are going to prove that there exist $x\in S_i$ and $y\in S_i\setminus\centre{S_i}$ such that $xy=yx$. Suppose that $\centre{S_i}\neq \emptyset$. Let $x\in\centre{S_i}$ and $y\in S_i\setminus\centre{S_i}$. We have $xy=yx$. Now suppose that $\commgraph{S_i}$ is not a null graph. Then there exist $x,y\in S_i\setminus\centre{S_i}$ such that $x$ and $y$ are adjacent in $\commgraph{S_i}$, which again implies that $xy=yx$. We can show in a similar way that there exist $z\in S_j$ and $w\in S_j\setminus\centre{S_j}$ such that $zw=wz$. Then condition 3 of Lemma~\ref{direct prod: lemma girth} is satisfied and, consequently, $\commgraph{S}$ contains a cycle (of length $3$).
	
	\smallskip
	
	\textit{Case 6:} Assume that condition 6 holds. It follows from the fact that $i\in\NC$ that there exists $x\in S_i\setminus\centre{S_i}$, and it follows from the fact that $\abs{\centre{S_i}}\geqslant 2$ that there exist distinct $y,z\in\centre{S_i}$. Then $x$, $y$ and $z$ commute with each other. Additionally, $\abs{\NC\setminus\set{i}}=\abs{\NC}-1\geqslant 1>0$ (because $i\in \NC$ and $\abs{\NC}\geqslant 2$). Thus, condition 1 of Lemma~\ref{direct prod: lemma girth} holds, which implies that $\commgraph{S}$ contains a cycle (of length $3$).
	
	\smallskip
	
	\textit{Case 7:} Assume that condition 7 holds. Since $\centre{S_i}\neq\emptyset$, then there exists $z\in\centre{S_i}$. Additionally, $\commgraph{S_i}$ is not a null graph, which implies that there exist distinct vertices $x,y\in S_i\setminus\centre{S_i}$ of $\commgraph{S_i}$ such that $x$ and $y$ are adjacent. Then $xy=yx$ and, as a consequence of the fact that $z\in\centre{S_i}$, we also have $xz=zx$ and $yz=zy$. Finally, we have $\abs{\NC\setminus\set{i}}=\abs{\NC}-1\geqslant 1>0$ (because $i\in \NC$ and $\abs{\NC}\geqslant 2$). Therefore, condition 1 of Lemma~\ref{direct prod: lemma girth} holds and, consequently, $\commgraph{S}$ contains a cycle (of length $3$).
	
	\smallskip
	
	\textit{Case 8:} Assume that condition 8 holds. Let $y_1-y_2-\cdots-y_m-y_1$ be a cycle in $\commgraph{S_i}$ and assume that $m=\girth{\commgraph{S_i}}$. Let $x_j\in S_j$ for all $j\in \setX\setminus\set{i}$. For each $k\in\X{m}$ let $s^{(k)}\in S$ be such that
	\begin{displaymath}
		s^{(k)}_j=\begin{cases}
			y_k& \text{if }j=i,\\
			x_j& \text{if }j\neq i
		\end{cases}
	\end{displaymath}
	for all $j\in \setX$. It follows from Lemma~\ref{direct prod: commutativity}, and the fact that $y_ky_{k+1}=y_{k+1}y_k$ for all $k\in\X{m-1}$, that $s^{(k)}s^{(k+1)}=s^{(k+1)}s^{(k)}$ for all $k\in\X{m-1}$. Additionally, since $y_1,\ldots,y_m\in S_i\setminus\centre{S_i}$, then we also have $s^{(1)},\ldots,s^{(m)}\in S\setminus\centre{S}$. Therefore $s^{(1)}-s^{(2)}-\cdots-s^{(m)}-s^{(1)}$ is a cycle in $\commgraph{S}$ and, consequently, $\girth{\commgraph{S}}\leqslant m =\girth{\commgraph{S_i}}$.
	
	\medskip
	
	\textbf{Part 2.} Assume that $\commgraph{S}$ contains cycles and that conditions 1--7 do not hold. Our aim is to prove that there exists $i\in \NC$ such that $\commgraph{S_i}$ contains cycles, that is, we want to see that condition 8 must hold. Let $s^{(1)}-s^{(2)}-\cdots-s^{(m)}-s^{(1)}$ be a cycle in $\commgraph{S}$ and assume that $m=\girth{\commgraph{S}}$.
	
	Let
	\begin{align*}
		&A_1=\gset{k\in \NC}{\centre{S_k}\neq\emptyset \text{ or } \commgraph{S_k} \text{ is  not a null graph}},\\
		&A_2=\gset{k\in C}{\abs{S_k}=2},\\
		&A_3=\gset{k\in C}{\abs{S_k}\geqslant 3}.
	\end{align*}
	
	Since conditions 1, 2 and 5 do not hold, then $\abs{A_3}=0$, $\abs{A_2}\leqslant 1$ and $\abs{A_1}\leqslant 1$. Consequently, there exist $j\in C$ and $i\in \NC$ such that $A_2\subseteq\set{j}$ and $A_1\subseteq \set{i}$.
	
	We have, by Lemma~\ref{direct prod: commutativity}, that $s^{(l)}_ks^{(l+1)}_k=s^{(l+1)}_ks^{(l)}_k$ for all $l\in\X{m-1}$ and $k\in \setX$. Since we also have $\centre{S_k}=\emptyset$ and $\commgraph{S_k}$ is a null graph (that is, $\commgraph{S_k}$ only contains isolated vertices) for all $k\in \NC\setminus\set{i}$, then we have $s^{(1)}_k=s^{(2)}_k=\cdots=s^{(m)}_k$ for all $k\in \NC\setminus\set{i}$. Moreover, we have $s^{(1)}_k=s^{(2)}_k=\cdots=s^{(m)}_k$ for all $k\in C\setminus\set{j}$ (because $\abs{S_k}=1$ for all $k\in C\setminus\set{j}$). Combined with the minimality of $m$, this implies that $\parens[\big]{s^{(1)}_i,s^{(1)}_j},\parens[\big]{s^{(2)}_i,s^{(2)}_j},\ldots,\parens[\big]{s^{(m)}_i,s^{(m)}_j}$ are pairwise distinct. Due to the fact that $\abs{S_j}\leqslant 2$ and $m\geqslant 3$, we have $\abs[\big]{\gset[\big]{s^{(l)}_i}{l\in\X{m}}}\geqslant 2$. In addition, since $s^{(l)}_is^{(l+1)}_i=s^{(l+1)}_is^{(l)}_i$ for all $l\in\X{m-1}$ (by Lemma~\ref{direct prod: commutativity}), we have $\centre{S_i}\neq \emptyset$ or $\commgraph{S_i}$ contains non-isolated vertices (that is, $\commgraph{S_i}$ is not a null graph), which implies that $A_1=\set{i}$. We consider the following two cases:
	
	\smallskip
	
	\textit{Case 1:} Assume that $\abs{\NC}=1$. Then $\NC=\set{i}$. It follows from the fact that $s^{(1)},\ldots,s^{(m)}\in S\setminus\centre{S}$, and Proposition~\ref{direct prod: Z(S), S comm <=> S_i comm}, that $s^{(1)}_i,\ldots,s^{(m)}_i\in S_i\setminus\centre{S_i}$, which implies that $\commgraph{S_i}$ contains non-isolated vertices (that is, $\commgraph{S_i}$ is not a null graph). Since condition 3 does not hold, then $A_2=\emptyset$ and we have $\abs{S_j}=1$. Hence $s^{(1)}_j=s^{(2)}_j=\cdots=s^{(m)}_j$ and, consequently, $s^{(1)}_i,\ldots,s^{(m)}_i$ are pairwise distinct. Furthermore, Lemma~\ref{direct prod: commutativity} implies that $s^{(1)}_is^{(m)}_i=s^{(m)}_is^{(1)}_i$ and $s^{(l)}_is^{(l+1)}_i=s^{(l+1)}_is^{(l)}_i$ for all $l\in\X{m-1}$. Thus $s^{(1)}_i-s^{(2)}_i-\cdots-s^{(m)}_i-s^{(1)}_i$ is a cycle in $\commgraph{S_i}$ and, consequently, $\girth{\commgraph{S_i}}\leqslant m=\girth{\commgraph{S}}$.
	
	\smallskip
	
	\textit{Case 2:} Assume that $\abs{\NC}\geqslant 2$. As a consequence of conditions 3 and 4 not holding, and the fact that $i\in A_1$, we have that $\abs{S_j}=1$. Thus $s^{(1)}_j=s^{(2)}_j=\cdots=s^{(m)}_j$ and, consequently, $s^{(1)}_i,\ldots,s^{(m)}_i$ are pairwise distinct. Additionally, we have $s^{(1)}_is^{(m)}_i=s^{(m)}_is^{(1)}_i$ and $s^{(l)}_is^{(l+1)}_i=s^{(l+1)}_is^{(l)}_i$ for all $l\in\X{m-1}$ (by Lemma~\ref{direct prod: commutativity}). Since $\abs{\centre{S_i}}\leqslant 1$ (because condition 6 does not hold) and $m\geqslant 3$, then we must have $s^{(1)}_i,s^{(m)}_i\in S_i\setminus\centre{S_i}$ or $s^{(t)}_i,s^{(t+1)}_i\in S_i\setminus\centre{S_i}$ for some $t\in\X{m-1}$. This implies that $\commgraph{S_i}$ contains non-isolated vertices (that is, $\commgraph{S_i}$ is not a null graph) and, since condition 7 does not hold, we have $\centre{S_i}=\emptyset$. Therefore $s^{(1)}_i,\ldots,s^{(m)}_i\in S_i\setminus\centre{S_i}$ and, consequently, $s^{(1)}_i-s^{(2)}_i-\cdots-s^{(m)}_i-s^{(1)}_i$ is a cycle in $\commgraph{S_i}$. In addition, we have $\girth{\commgraph{S_i}}\leqslant m=\girth{\commgraph{S}}$.

	\medskip
	
	\textbf{Part 3.} Now we determine $\girth{\commgraph{S}}$ when $\commgraph{S}$ contains cycles. It follows from cases 1--7 of part 1 of the proof that, when at least one of the conditions 1--7 is satisfied, then $\girth{\commgraph{S}}=3$.
	
	Assume that conditions 1--7 do not hold and that condition 8 holds. Then, by part 2 of the proof, there exists $i\in \NC$ such that $\commgraph{S_i}$ contains cycles and $\commgraph{S_k}$ is a null graph for all $k\in \NC\setminus\set{i}$ (which implies that $\commgraph{S_k}$ does not contain cycles for all $k\in \NC\setminus\set{i}$). Furthermore, we saw that $\girth{\commgraph{S_i}}\leqslant\girth{\commgraph{S}}$. In addition, it follows from the proof of case 8 of part 1, that $\girth{\commgraph{S}}\leqslant\girth{\commgraph{S_i}}$, which concludes the proof.
\end{proof}

The last result of this section concerns left paths. We are going to see that the existence of left paths in $\commgraph{S}$ does not depend uniquely on the existence of left paths in $\commgraph{S_i}$ for all $i\in\NC$ --- it also depends on the existence of $*$-left paths in $\extendedcommgraph{S_i}$ for all $i\in C$. Additionally, when $\commgraph{S}$ contains left paths, we supply a way to determine the knit degree of $S$.

\begin{theorem}
	Suppose that $\NC\neq\emptyset$. We have that $\commgraph{S}$ contains left paths if and only if at least one of the following conditions is satisfied:
	\begin{enumerate}
		\item There exists $i\in \NC$ such that $\commgraph{S_i}$ contains left paths.
		
		\item There exists $i\in \setX$ such that $\NC\setminus\set{i}\neq\emptyset$ and $\extendedcommgraph{S_i}$ contains $*$-left paths.
	\end{enumerate}
	Moreover, when $\commgraph{S}$ has left paths we have $\knitdegree{S}=\min \parens{K\cup K^*}$, where
	\begin{align*}
		K&=\gset{\knitdegree{S_i}}{i\in \NC \text{ and } \commgraph{S_i} \text{ contains left paths}},\\
		K^*&=\gset{\starknitdegree{S_i}}{i\in \setX \text{ and } \NC\setminus\set{i}\neq\emptyset \text{ and } \extendedcommgraph{S_i} \text{ contains $*$-left paths}}.
	\end{align*}
		
		
\end{theorem}

\begin{proof}
	\textbf{Part 1.} We begin by proving the forward implication. Suppose that $\commgraph{S}$ contains left paths. Let $s^{(1)}-s^{(2)}-\cdots-s^{(m)}$ be a left path in $\commgraph{S}$ and assume that $\knitdegree{S}=m-1$. Since $s^{(1)}\neq s^{(m)}$, then there exists $i\in \setX$ such that $s^{(1)}_i\neq s^{(m)}_i$. Furthermore, we have $s^{(1)}_is^{(k)}_i= \parens[\big]{s^{(1)}s^{(k)}}_i= \parens[\big]{s^{(m)}s^{(k)}}_i= s^{(m)}_is^{(k)}_i$ for all $k\in\X{m}$. This implies that, any path from $s^{(1)}_i$ to $s^{(m)}_i$ in $\commgraph{S_i}$ (respectively, $\extendedcommgraph{S_i}$) whose vertices belong to $\gset{s^{(k)}_i}{k\in\X{m}}$, is a left path of $\commgraph{S_i}$ (respectively, $*$-left path of $\extendedcommgraph{S_i}$). We will prove that such a path exists in $\commgraph{S_i}$ or in $\extendedcommgraph{S_i}$. We consider the following two cases.
	
	\smallskip
	
	\textit{Case 1:} Suppose that $s^{(1)}_i,\ldots,s^{(m)}_i\in S_i\setminus\centre{S_i}$. Then $i\in\NC$. It follows from the fact that $s^{(k)}s^{(k+1)}=s^{(k+1)}s^{(k)}$ for all $k\in\X{m-1}$, and Lemma~\ref{direct prod: commutativity}, that $s^{(k)}_is^{(k+1)}_i=s^{(k+1)}_is^{(k)}_i$ for all $k\in\X{m-1}$. Then $s^{(1)} \sim s^{(2)} \sim \cdots \sim s^{(m)}$ (in $\commgraph{S_i}$). If there exist $l,t\in\X{m}$ such that $l<t$ and $s^{(l)}_i=s^{(t)}_i$, then we have $s^{(l)}_i \sim s^{(t+1)}_i$ (because $s^{(t)}_i \sim s^{(t+1)}_i$), which means we can suppress $s^{(l+1)}_i,s^{(l+2)}_i,\ldots,s^{(t)}_i$ from the sequence of vertices and obtain the new one $s^{(1)}_i \sim s^{(2)}_i \sim \cdots\sim s^{(l)}_i \sim s^{(t+1)}_i\sim \cdots \sim s^{(m)}_i$. (We observe that we might have $t=m$. In that case the new sequence is $s^{(1)}_i \sim s^{(2)}_i \sim \cdots \sim s^{(l)}_i=s^{(m)}_i$.) We can repeat this process until we obtain a sequence of pairwise distinct vertices. This sequence forms a path from $s^{(1)}_i$ to $s^{(m)}_i$ in $\commgraph{S_i}$ whose vertices belong to $\gset{s^{(k)}_i}{k\in\X{m}}$. Thus $\commgraph{S}$ contains a left path (whose length is at most $m-1$) and we have $\knitdegree{S_i}\leqslant m-1=\knitdegree{S}$.
	
	\smallskip
	
	\textit{Case 2:} Suppose that there exists $l\in\X{m}$ such that $s^{(l)}_i\in\centre{S_i}$. Assume that $s^{(l)}_i=s^{(1)}_i$ or $s^{(l)}_i=s^{(m)}_i$. Then $s^{(1)}_i\in\centre{S_i}$ or $s^{(m)}_i\centre{S_i}$ and, consequently, $s^{(1)}_is^{(m)}_i=s^{(m)}_is^{(1)}_i$. Thus $s^{(1)}_i-s^{(m)}_i$ is a $*$-left path (of length $1$) in $\commgraph{S}$ and, consequently, $\starknitdegree{S_i}=1\leqslant m-1=\knitdegree{S}$. Now assume that $s^{(l)}_i\neq s^{(1)}_i$ and $s^{(l)}_i\neq s^{(m)}_i$. We have $s^{(1)}_is^{(l)}_i=s^{(l)}_is^{(1)}_i$ and $s^{(l)}_is^{(m)}_i=s^{(m)}_is^{(l)}_i$ (because $s^{(l)}_i\in \centre{S_i}$). Thus $s^{(1)}_i-s^{(l)}_i-s^{(m)}_i$ is a $*$-left path in $\extendedcommgraph{S_i}$ (of length $2$). In addition, since $s^{(1)}_i$, $s^{(l)}_i$ and $s^{(m)}_i$ are pairwise distinct, then $m\geqslant 3$ and, consequently, $\starknitdegree{S_i}\leqslant 2\leqslant m-1=\knitdegree{S}$.
	
	\medskip
	
	\textbf{Part 2.} Now we are going to prove the reverse implication. We consider the following two cases:
	
	\smallskip
	
	\textit{Case 1:} Suppose that there exists $i\in \NC$ such that $\commgraph{S_i}$ contains left paths. Let $x_1-x_2-\cdots-x_m$ be a left path in $\commgraph{S_i}$ and assume that $\knitdegree{S_i}=m-1$. For each $j\in\NC\setminus\set{i}$ we select $y_j\in S_j\setminus\centre{S_j}$ and for each $j\in C$ we select $z_j\in S_j$. Let $s^{(1)},\ldots,s^{(m)}\in S$ be such that
	\begin{displaymath}
		s^{(k)}_j=\begin{cases}
			x_k& \text{if } j=i,\\
			y_j& \text{if } j\in\NC\setminus\set{i},\\
			z_j& \text{if } j\in C
		\end{cases}
	\end{displaymath}
	for all $k\in\X{m}$ and $j\in \setX$. We note that, since $x_1,\ldots,x_m\in S_i\setminus\centre{S_i}$, then we also have $s^{(1)},\ldots,s^{(m)}\in S\setminus\centre{S}$ (by Proposition~\ref{direct prod: Z(S), S comm <=> S_i comm}). It follows from Lemma~\ref{direct prod: commutativity}, and the fact that $x_{k}x_{k+1}=x_{k+1}x_{k}$ for all $k\in\X{m-1}$, that $s^{(k)}s^{(k+1)}=s^{(k+1)}s^{(k)}$ for all $k\in\X{m-1}$. Hence $s^{(1)}-s^{(2)}-\cdots-s^{(m)}$ is a path in $\commgraph{S}$. Moreover, $s^{(1)}\neq s^{(m)}$ (because $x_1\neq x_m$), and for all $k\in\X{m}$ and $j\in \setX$ we have
	\begin{align*}
		\parens[\big]{s^{(1)}s^{(k)}}_j&=s^{(1)}_js^{(k)}_j\\
		&=\begin{cases}
			x_1x_k& \text{if } j=i,\\
			y_jy_j& \text{if } j\neq i
		\end{cases}\\
		&=\begin{cases}
			x_mx_k& \text{if } j=i,\\
			y_jy_j& \text{if } j\neq i
		\end{cases}\\
		&=s^{(m)}_js^{(k)}_j\\
		&=\parens[\big]{s^{(m)}s^{(k)}}_j,
	\end{align*}
	which implies that $s^{(1)}s^{(k)}=s^{(m)}s^{(k)}$ for all $k\in\X{m}$. Thus $s^{(1)}-s^{(2)}-\cdots-s^{(m)}$ is a left path in $\commgraph{S}$ and we have $\knitdegree{S}\leqslant m-1=\knitdegree{S}$.
	
	\smallskip
	
	\textit{Case 2:} Suppose that there exists $i\in\setX$ such that $\NC\setminus\set{i}\neq\emptyset$ and $\extendedcommgraph{S_i}$ contains $*$-left paths. Let $t\in\NC\setminus\set{i}$ and let $x_1-x_2-\cdots-x_m$ be a $*$-left path in $\commgraph{S_i}$. The proof of this case is similar to that of case 1. The main difference is the following: unlike the previous case, we might have $\gset{x_k}{k\in\X{m}}\cap\centre{S_i}\neq\emptyset$. Thus, what justifies the conclusion that $s^{(1)},\ldots,s^{(m)}\in S\setminus \centre{S}$ is the fact that $s^{(1)}_t=s^{(2)}_t=\cdots=s^{(m)}_t=y_t\in S_t\setminus \centre{S_t}$, together with Proposition~\ref{direct prod: Z(S), S comm <=> S_i comm}. (We observe that in the previous case we could have $\NC\setminus\set{i}=\emptyset$.) We can also obtain in a similar way that $\knitdegree{S}\leqslant\starknitdegree{S_i}$.
	
	\medskip
	
	\textbf{Part 3.} Now we determine $\knitdegree{S}$ when $\commgraph{S}$ contains left paths. We note that, as a consequence of part 1 of the proof, $K\cup K^*\neq\emptyset$. Furthermore, it also follows from part 1 of the proof that there exists $i\in\NC$ such that $\commgraph{S_i}$ contains left paths and $\knitdegree{S}\geqslant\knitdegree{S_i}\geqslant\min\parens{K\cup K^*}$; or there exists $i\in \setX$ such that $\extendedcommgraph{S_i}$ contains $*$-left paths, $\NC\setminus\set{i}\neq\emptyset$ and $\knitdegree{S}\geqslant\starknitdegree{S_i}\geqslant\min\parens{K\cup K^*}$. Additionally, part 2 of the proof implies that for all $i\in \NC$ such that $\commgraph{S_i}$ contains left paths we have $\knitdegree{S}\leqslant{\knitdegree{S_i}}$; and that for all $i\in \setX$ such that $\NC\setminus\set{i}\neq\emptyset$ and $\extendedcommgraph{S_i}$ contains $*$-left paths we have $\knitdegree{S}\leqslant\starknitdegree{S_i}$. Thus $\knitdegree{S}\leqslant\min\parens{K\cup K^*}$, which concludes the proof.
\end{proof}

    \bibliography{Bibliography} 
\bibliographystyle{alphaurl}

\end{document}